\documentclass{article}

\usepackage[utf8]{inputenc}
\usepackage[T1]{fontenc}
\usepackage{graphicx}
\usepackage{tikz}
\usepackage{amsmath,amsfonts,bm,amsthm}
\usepackage{multirow}
\usepackage[pdftex,colorlinks=true,linkcolor=blue,citecolor=blue,urlcolor=blue]{hyperref}

\DeclareMathOperator{\diag}{diag}
\newcommand{\x}{\mathbf{x}}
\newcommand{\R}{\mathbb{R}}
\renewcommand{\d}{\mathrm{d}}
\let\L\relax
\DeclareMathOperator{\L}{\mathcal{L}}
\newcommand{\y}{\bm{y}}
\newcommand{\lam}{\bm{\lambda}}

\newtheorem{remark}{Remark}
\newtheorem{theorem}{Theorem}
\newtheorem{corollary}{Corollary}

\title{New time domain decomposition methods for parabolic control problems I: Dirichlet-Neumann and Neumann-Dirichlet algorithms}

\author{Martin Jakob Gander$^{1}$, Liu-Di LU$^{1}$}
\date
{%
\noindent{\small\textit{$^1$Section of Mathematics, University of Geneva, Rue du Conseil Général 7-9, 1205 Geneva, Switzerland}}\\
}

\begin{document}

\maketitle

\begin{abstract} 
We present new Dirichlet-Neumann and Neumann-Dirichlet algorithms with a time domain
 decomposition applied to unconstrained parabolic optimal control problems. After a spatial 
 semi-discretization, we use the Lagrange multiplier approach to derive a coupled 
 forward-backward optimality system, which can then be solved using a time domain
 decomposition. Due to the forward-backward structure of the optimality system, three
 variants can be found for the Dirichlet-Neumann and Neumann-Dirichlet algorithms. We
 analyze their convergence behavior and determine the optimal relaxation parameter for
 each algorithm. Our analysis reveals that the most natural algorithms are actually only 
 good smoothers, and there are better choices which lead to efficient solvers. We illustrate
 our analysis with numerical experiments.  
\end{abstract}

\begin{paragraph}{Keywords: }
Time domain decomposition, Dirichlet-Neumann algorithm, Neumann-Dirichlet algorithm, Parabolic optimal control problems, Convergence analysis.
\end{paragraph}

\begin{paragraph}{MSCcodes: }
65M12, 65M55, 65Y05,
\end{paragraph}

\section{Introduction}
PDE-constrained optimal control problems arise in various areas, often containing multiphysics or multiscale phenomena, and also high frequency components on different time scales. This requires very fine spatial and temporal discretizations, resulting in very large problems, for which efficient parallel solvers are needed; we refer to~\cite{Hinze2009,Troltzsch2010} for a brief review. We present and analyze a new class of time domain decomposition methods based on Dirichlet-Neumann and Neumann-Dirichlet techniques. We consider as our model a parabolic control problem: for a given target function $\hat y\in L^2(Q)$ and $\gamma\in\R^+,\nu\in\R^+_*$, we want to minimize the cost functional
\begin{equation}\label{eq:J}
J(y,u) := \frac12 \|y -\hat y\|^2_{L^2(Q)} \,\d t + \frac{\gamma}2 \|y(T) - \hat y(T)\|^2_{L^2(\Omega)} + \frac{\nu}2\|u\|^2_{U_{\text{ad}}},
\end{equation}
subject to the linear parabolic state equation
  \begin{equation}\label{eq:heat}
    \begin{aligned}
      \partial_t y - \Delta y &= u \quad &&\text{ in } Q:=\Omega\times (0,T), \\
      y &=0 &&\text{ on }\Sigma:=\partial\Omega\times(0,T), \\
      y(0) &=y_0 &&\text{ on }\Sigma_0:=\Omega\times\{0\},
    \end{aligned}
  \end{equation}
where $\Omega\subset\R^d$, $d=1,2,3$ is a bounded domain with boundary $\partial\Omega$, and $T$ is the fixed final time. The control $u$ on the right-hand side of the PDE is in an admissible set $U_{\text{ad}}$, and we want to control the solution of the parabolic PDE~\eqref{eq:heat} towards a target state $\hat y$. For simplicity, we consider here homogeneous boundary conditions.

  The parabolic control problem~\eqref{eq:J}-\eqref{eq:heat} has a
  unique solution for the classical choice $u\in L^2(Q)$, which can
  be characterized by a forward-backward optimality system, see e.g.
  \cite{Borzi2011,Lions1971,Troltzsch2010}. More recently, also
  energy regularization has been considered,
  see~\cite{Neumuller2021} for elliptic and~\cite{Langer2021675} for
  parabolic cases. This is motivated by the fact that the state
  $y\in L^2(0,T;H^1_0(\Omega))$ is well-defined as the solution of the
  heat equation~\eqref{eq:heat} for the control $z\in
  L^2(0,T;H^{-1}(\Omega))$, and thus offers an interesting alternative.
  
  We are interested in applying Time Domain Decomposition methods
(DDMs) to the forward-backward optimality system. DDMs were 
developed for elliptic PDEs and are very efficient in parallel
  computing environments, see e.g.~\cite{Dolean2015,Toselli2005}. DDMs
  were extended to time-dependent problems using waveform relaxation
  techniques from~\cite{Lelarasmee1982}, with a spatial
decomposition and solving the problem on small space-time
cylinders~\cite{Gander1998}. The extension of DDMs to elliptic control
problems is quite natural,
see~\cite{Bartlett2006,Benamou1996,Chang2011,Gander2019}, but less
  is known about DDMs applied to parabolic control problems.

The role of the time variable in forward-backward optimality systems
is key, and it is natural to seek efficient solvers through time
domain decomposition. For classical evolution problems, the idea
of time domain decomposition goes back
  to~\cite{Nievergelt1964}. Parallel Runge Kutta methods were
introduced in~\cite{Miranker1967} with good small scale time
parallelism. In~\cite{Lubich1987,Vandewalle1994}, the authors propose
to combine multigrid methods with waveform
relaxation. Parareal~\cite{Lions2002} uses a different approach,
  namely multiple shooting with an approximate Jacobian on a coarse
grid, and Parareal techniques led to a new ParaOpt
  algorithm~\cite{Gander2020} for optimal control, see also
  \cite{Heinkenschloss2005}. In~\cite{Gander2016,Kwok2017}, Schwarz
  methods are used to decompose the time domain for
optimal control.

We develop and analyze here new time domain decomposition algorithms to solve
  the PDE-constrained problem~\eqref{eq:J}-\eqref{eq:heat} using
  Dirichlet-Neumann and Neumann-Dirichlet techniques that go back
  to~\cite{Bjorstad1986} for space parallelism. We introduce in
Section~\ref{sec:2} the optimality system and its
  semi-discretization. In Section~\ref{sec:3} we present our new time
  parallel Dirichlet-Neumann and Neumann-Dirichlet algorithms and
  study their convergence. Numerical experiments are shown in
Section~\ref{sec:4}, and we draw conclusions in Section~\ref{sec:5}.

\section{Optimality system and its semi-discretization}\label{sec:2}

	The PDE-constrained optimization problem~\eqref{eq:J}-\eqref{eq:heat} 
	can be solved using Lagrange multipliers~\cite[Chapter 3]{Troltzsch2010}, see
  	also~\cite{Gander2014} for a historical context. To obtain the
  	associated optimality system, we introduce the Lagrangian function $\L$
	associated with Problem~\eqref{eq:J}-\eqref{eq:heat},
	\[\begin{aligned}
		\L(y,u,\lambda) = &J(y,u) + \langle \partial_t y - \Delta y - u,\lambda\rangle\\
		= &\int_0^T\Big(\langle \partial_t y,\lambda\rangle_{V',V} + \int_{\Omega}\big(\frac12|y-\hat y|^2+\frac{\nu}2|u|^2 + \nabla y\cdot\nabla\lambda - u\lambda\big) \,\d\x\Big)\,\d t\\
		& + \frac{\gamma}2\int_{\Omega} |y(T) - \hat y(T)|^2\,\d \x,
	\end{aligned}\]
with $y\in W(0,T):=L^2(0,T;V)\cap H^1(0,T;V')$, $u\in L^2(Q)$,
$V:=H^1_0(\Omega)$ and $V':=H^{-1}(\Omega)$ the dual space of
$V$. Here $\lambda\in L^2(0,T;V)$ denotes the adjoint state (also
called the Lagrange multiplier). Taking the derivative of
$\L$ with respect to $\lambda$ and equating this to
zero, we find for all test functions $\chi\in L^2(0,T;V)$,
\[0=\langle\partial_{\lambda}\L(y,u,\lambda),\chi\rangle = \int_0^T\big(\langle \partial_t y,\chi\rangle_{V',V}+ \int_{\Omega}\big( \nabla y\cdot\nabla\chi - u\chi\big) \,\d\x\big)\,\d t,\]
which implies that $y\in V$ is the weak solution of the state
equation~\eqref{eq:heat} (also called the primal problem).
Taking the derivative of $\L$ with respect to $y$ and
equating this to zero, and obtain for all $\chi\in W(0,T)$
\[\begin{aligned}
  0=\langle\partial_{y}\L(y,u,\lambda),\chi\rangle = &\int_0^T\Big(\langle \partial_t \chi,\lambda\rangle_{V',V}+ \int_{\Omega}\big( (y-\hat y)\chi + \nabla \chi\cdot\nabla\lambda \big) \,\d\x\Big)\,\d t\\
  =&\langle \chi(T),\lambda(T)+\gamma(y(T)-\hat y(T)) \rangle_{L^2(\Omega)} - \langle \chi(0),\lambda(0) \rangle_{L^2(\Omega)}\\
  &+\int_0^T\langle-\partial_t \lambda-\Delta \lambda + (y-\hat y) ,\chi\rangle_{V',V}\,\d t,
\end{aligned}\]
where we used integration by parts with respect to $t$ in $\partial_t \chi$ and with respect to $\x$ in $\nabla \chi$. By choosing $\chi\in C_0^{\infty}(Q)$ and applying an argument of density, we find that the last integral is zero. Choosing then $\chi\in W(0,T)$ such that $\chi(0)=0$, we obtain the adjoint equation (also called the dual problem)
\begin{equation}\label{eq:adjoint}
  \begin{aligned}
    \partial_t \lambda + \Delta \lambda &= y - \hat y \quad &&\text{ in } Q, \\
    \lambda &=0 &&\text{ on }\Sigma, \\
    \lambda(T) &=-\gamma(y(T)-\hat y(T)) &&\text{ on }\Sigma_T:=\Omega\times\{T\}.
  \end{aligned}
\end{equation}
Finally, taking the derivative of $\L$ with respect to $u$ and equating
this to zero, we obtain for all test functions $\chi\in L^2(Q)$, 
%$0=\langle \partial_u(y,u,p),\chi\rangle = \int_0^T\int_{\Omega} (\nu u - \lambda) \chi \, \d \x \, \d t$,
\[0=\langle \partial_u(y,u,p),\chi\rangle = \int_0^T\int_{\Omega} (\nu u - \lambda) \chi \, \d \x \, \d t,\]
which gives the optimality condition 
\begin{equation}\label{eq:opt_cond}
  \lambda = \nu u \quad \text{ in } Q.
\end{equation}
If a control $u$ is optimal with the associated state $y$ of the
optimization problem~\eqref{eq:J}-\eqref{eq:heat}, then the
first-order optimality system~\eqref{eq:heat}, \eqref{eq:adjoint}
and~\eqref{eq:opt_cond} must be satisfied. This is a
forward-backward system, i.e., the primal problem is solved forward
in time with an initial condition while the dual problem is solved
backward in time with a final condition, and our new time decomposition
  algorithms solve this system. Since the time variable plays a
special role, we consider a semi-discretization in space, and
replace the spatial operator $-\Delta$ in the primal
problem~\eqref{eq:heat} by a matrix $A\in\R^{n\times n}$, for
instance using a Finite Difference discretization in
space. We then obtain as above the semi-discrete optimality system
(dot denoting the time derivative)
\[
  \left\{
    \begin{aligned}
      \dot \y + A \y &= \bm{u} \quad \text{ in } (0,T),\\
      \y(0) &= \y_0,
    \end{aligned}
  \right.
  \qquad
  \left\{
    \begin{aligned}
      \dot \lam - A^T \lam &= \y-\hat \y \quad \text{ in } (0,T),\\
      \lam(T) &= -\gamma(\y(T)-\hat \y(T)),
    \end{aligned}
  \right.
\]
where $\lam(t) = \nu \bm{u}(t)$ for all $t\in\Omega$. Eliminating
$\bm{u}$, we obtain in matrix form 
\begin{equation}\label{eq:sysODE}
  \left\{
    \begin{aligned}
      \begin{pmatrix}
        \dot \y\\
        \dot \lam
      \end{pmatrix}
      +
      \begin{pmatrix}
        A & -\nu^{-1} I\\
        -I & -A^T
      \end{pmatrix}
      \begin{pmatrix}
        \y\\
        \lam
      \end{pmatrix}
      &=
      \begin{pmatrix}
        0\\
        -\hat \y
      \end{pmatrix} \text{ in } (0,T),\\
      \y(0) &= \y_0,\\
      \lam(T)+\gamma \y(T) &= \gamma\hat \y(T),
    \end{aligned}
  \right.
\end{equation}
where $I$ is the identity. If $A$ is symmetric, $A= A^T$, 
which is natural for discretizations of $-\Delta$, then it
can be diagonalized, $A=PDP^{-1}$, $D:=\diag(d_1,\ldots,d_n)$
 with $d_i$ the $i$-th eigenvalue of $A$. The
system~\eqref{eq:sysODE} can thus also be diagonalized
\[
  \left\{\begin{aligned}
    \begin{pmatrix}
      \dot{\bm{z}}\\
      \dot{\bm{\mu}}
    \end{pmatrix}
    +
    \begin{pmatrix}
      D & -\nu^{-1} I\\
      -I & -D
    \end{pmatrix}
    \begin{pmatrix}
      \bm{z}\\
      \bm{\mu}
    \end{pmatrix}
    &=
    \begin{pmatrix}
      0\\
      -\hat{\bm{z}}
    \end{pmatrix} \text{ in } (0,T),\\
    \bm{z}(0) &= \bm{z}_0,\\
    \bm{\mu}(T) +\gamma \bm{z}(T)&= \gamma\hat{\bm{z}}(T),
  \end{aligned}\right.
\]
where $\bm{z} := P^{-1} \y$, $\mu:=P^{-1}\lam$, $\hat{\bm{z}}:=P^{-1}
\hat \y$ and $\bm{z}_0:=P^{-1}\y_0$. This system then represents
  $n$ independent $2\times 2$ systems of ODEs of the form
\begin{equation}\label{eq:sysODEreduced}
  \left\{
    \begin{aligned}
      \begin{pmatrix}
        \dot z_{(i)}\\
        \dot \mu_{(i)}
      \end{pmatrix}
      +
      \begin{pmatrix}
        d_i & -\nu^{-1} \\
        -1 & -d_i
      \end{pmatrix}
      \begin{pmatrix}
        z_{(i)}\\
        \mu_{(i)}
      \end{pmatrix}
      &=
      \begin{pmatrix}
        0\\
        -\hat z_{(i)}
      \end{pmatrix} \text{ in } (0,T),\\
      z_{(i)}(0) &= z_{(i),0},\\
      \mu_{(i)}(T) + \gamma z_{(i)}(T) &= \gamma \hat z_{(i)}(T),
    \end{aligned}
  \right.
\end{equation}
where $z_{(i)}$, $\mu_{(i)}$, $\hat z_{(i)}$ are the $i$-th components of the vectors
$\bm{z}$, $\bm{\mu}$, $\hat{\bm{z}}$. Isolating
  the variable in each equation in~\eqref{eq:sysODEreduced}, we find
  the identities
\begin{equation}\label{eq:2id} 
  \mu_{(i)} = \nu(\dot z_{(i)} + d_i z_{(i)}), \qquad z_{(i)} = \dot \mu_{(i)}  - d_i \mu_{(i)} +\hat z_{(i)}.
\end{equation}
%We can then obtain second-order ODEs from~\eqref{eq:sysODEreduced}: using the identity of $z$ to eliminate $\mu$, we find
We use the identity of $z$ to eliminate $\mu$, and obtain a second-order ODE from~\eqref{eq:sysODEreduced},
\begin{equation}\label{eq:z}
  \left\{\begin{aligned}
    \ddot z_{(i)}- (d_i^2+\nu^{-1}) z_{(i)} &= -\nu^{-1}\hat z_{(i)} \text{ in } (0,T),\\
    z_{(i)}(0) &= z_{(i),0},\\
    \dot z_{(i)}(T) + (\nu^{-1}\gamma + d_i) z_{(i)}(T)&= \nu^{-1}\gamma \hat z_{(i)}(T).
  \end{aligned}\right.
\end{equation}
Similarly, we can also eliminate $z$ to get
\begin{equation}\label{eq:mu}
  \left\{\begin{aligned}
    \ddot \mu_{(i)}- (d_i^2+\nu^{-1}) \mu_{(i)} &= -\dot{\hat z}_{(i)}-d_i\hat z_{(i)} \text{ in } (0,T),\\
    \dot \mu_{(i)}(0)-d_i\mu_{(i)}(0) &= z_{(i),0} -\hat z_{(i)}(0),\\
    \gamma \dot \mu_{(i)}(T)+(1-\gamma d_i)\mu_{(i)}(T)&=0.
  \end{aligned}\right.
\end{equation}
To simplify the notation in what follows, we define
\begin{equation}\label{eq:sob}
  \sigma_i:=\sqrt{d_i^2+\nu^{-1}}, \quad \omega_i:=\nu^{-1}\gamma + d_i, \quad \beta_i:=1-\gamma d_i.
\end{equation}
In our analysis for the error, $\hat \y$ will equal zero, which
  implies $\hat{\bm{z}}=0$, and the solution of~\eqref{eq:z}
  and~\eqref{eq:mu} is then
\begin{equation}\label{eq:gensol}
  z_{(i)}(t) \text{ or } \mu_{(i)}(t) \, = A_i\cosh(\sigma_i t) + B_i\sinh(\sigma_i t),
\end{equation}
where $A_i, B_i$ are two coefficients.
\begin{remark}
  	Our arguments above work for any diagonalizable matrix $A$,
    and thus our results will apply to more general parabolic optimal
    control problems than the heat equation. Note also that the
    diagonalization is only a theoretical tool for our convergence
    analysis, and not needed to run our new time domain decomposition algorithms.
\end{remark}

\section{Dirichlet-Neumann and Neumann-Dirichlet algorithms in time}\label{sec:3}

  We now apply Dirichlet-Neumann (DN) and Neumann-Dirichlet (ND)
  techniques in time to obtain our new time domain decomposition algorithms to
  solve the system~\eqref{eq:sysODEreduced}, and study their
  convergence. Focusing on the error equations, we set the initial
condition $\y_0=0$ (i.e., $\bm{z}_0=0$) and the target functions $\hat
\y=0$ (i.e., $\hat{\bm{z}}=0$). We decompose the time domain
$\Omega:=(0,T)$ into two non-overlapping time subdomains
$\Omega_1:=(0,\alpha)$ and $\Omega_2:=(\alpha,T)$, where $\alpha$ is
the interface. We denote by $z_{j,(i)}$ and $\mu_{j,(i)}$ the
restriction to $\Omega_j$, $j=1,2$ of $z_{(i)}$ and $\mu_{(i)}$.
Since system~\eqref{eq:sysODEreduced} is a forward-backward
system, it appears natural at first sight to keep this property
for the decomposed case, as illustrated in
Figure~\ref{fig:forback}:
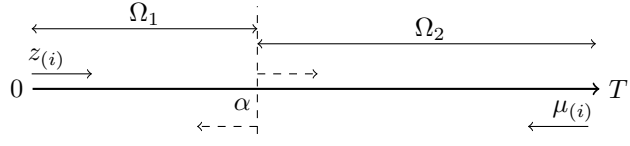
\begin{figure}
  \centering
  \begin{tikzpicture}
    \node (a) at (0,0) {0};
    \node (b) at (8,0) {$T$};
    \node (c) at (3,-0.2) {$\alpha$};
    \draw[thick,->] (a)--(b);
    \draw[dashed] (3.2,-0.6)--(3.2,1.1);
    \draw[<->] (0.2,0.8)--(3.2,0.8);
    \draw[<->] (3.2,0.6)--(7.7,0.6);
    \node (d) at (1.7,1) {$\Omega_1$};
    \node (e) at (5.5,0.8) {$\Omega_2$};
    \draw[->] (0.2,0.2)--(1,0.2);
    \node at (0.4,0.4) {$z_{(i)}$};
    \draw[dashed,->] (3.2,0.2)--(4,0.2);
    \draw[->] (7.6,-0.5)--(6.8,-0.5);
    \node at (7.4,-0.3) {$\mu_{(i)}$};
    \draw[dashed,->] (3.2,-0.5)--(2.4,-0.5);
  \end{tikzpicture}
  \caption{Illustration of the forward-backward system.}
  \label{fig:forback}
\end{figure}
we expect to have a final condition for the adjoint state $\mu_{(i)}$
in $\Omega_1$ since we already have an initial condition for
$z_{(i)}$; similarly, we expect to have an initial condition for the
primal state $z_{(i)}$ in $\Omega_2$ since we already have a final
condition for $\mu_{(i)}$. Therefore, a natural DN 
  algorithm in time solves for the iteration index $k=1,2,\ldots$
\begin{equation}\label{eq:DN1}
  \begin{aligned}
    &\left\{
      \begin{aligned}
        \begin{pmatrix}
          \dot z_{1,(i)}^k\\
          \dot \mu_{1,(i)}^k
        \end{pmatrix}
        +
        \begin{pmatrix}
          d_i & -\nu^{-1} \\
          -1 & -d_i
        \end{pmatrix}
        \begin{pmatrix}
          z_{1,(i)}^k\\
          \mu_{1,(i)}^k
        \end{pmatrix}
        &=
        \begin{pmatrix}
          0\\
          0
        \end{pmatrix} \text{ in } \Omega_1,\\
        z_{1,(i)}^k(0) &= 0,\\
        \mu_{1,(i)}^k(\alpha) & = f_{\alpha,(i)}^{k-1},
      \end{aligned}
    \right.\\
    &\left\{
      \begin{aligned}
        \begin{pmatrix}
          \dot z_{2,(i)}^k\\
          \dot \mu_{2,(i)}^k
        \end{pmatrix}
        +
        \begin{pmatrix}
          d_i & -\nu^{-1} \\
          -1 & -d_i
        \end{pmatrix}
        \begin{pmatrix}
          z_{2,(i)}^k\\
          \mu_{2,(i)}^k
        \end{pmatrix}
        &=
        \begin{pmatrix}
          0\\
          0
        \end{pmatrix} \text{ in } \Omega_2,\\
        \dot z_{2,(i)}^k(\alpha) &= \dot z_{1,(i)}^k(\alpha),\\
        \mu_{2,(i)}^k(T) + \gamma z_{2,(i)}^k(T)&=0,
      \end{aligned}
    \right.
  \end{aligned}
\end{equation}
and then the transmission condition is updated by
\begin{equation}\label{eq:DN1tran}
  f_{\alpha,(i)}^{k}:= (1-\theta)f_{\alpha,(i)}^{k-1} + \theta\mu_{2,(i)}^k(\alpha),
\end{equation}
with a relaxation parameter $\theta\in(0,1)$. However, there are
many other ways to decouple in time using DN and ND techniques for
  problem~\eqref{eq:sysODEreduced}: we can apply the technique to
both states $(z_{(i)},\mu_{(i)})$ as in~\eqref{eq:DN1}, or we can
apply it just to one of these two states in the reduced
  forms~\eqref{eq:z} and~\eqref{eq:mu}.
And with the identities~\eqref{eq:2id}, we can transfer
the Dirichlet and the Neumann transmission condition from one state to
the other. We list in Table~\ref{tab:combination} all possible new
  time domain decomposition algorithms we can obtain, along with
their equivalent representations in terms of other formulations.
\begin{table}
  \begin{center}
    \begin{tabular}{ c|c|c|c|c }
     &Problem &$\Omega_1$ &$\Omega_2$& algorithm type  \\
     \hline\hline 
     \multirow{6}{*}{Category I: $(z_{(i)},\mu_{(i)})$}&\eqref{eq:sysODEreduced} &$\mu_{(i)}$ &$\dot z_{(i)}$ &(DN) \\  
     &\eqref{eq:z} &$\dot z_{(i)}+d_i z_{(i)}$ &$\dot z_{(i)}$ &(RN)\\
     & \eqref{eq:mu} &$\mu_{(i)}$ &$\ddot \mu_{(i)}-d_i\dot\mu_{(i)}$ &(DR) \\
     \cline{2-5} 
     &\eqref{eq:sysODEreduced} &$\dot \mu_{(i)}$ &$z_{(i)}$ &(ND)\\  
     &\eqref{eq:z} &$\ddot z_{(i)}+d_i \dot z_{(i)}$ &$z_{(i)}$ &(RD)\\ 
     &\eqref{eq:mu} &$\dot \mu_{(i)}$ &$\dot \mu_{(i)}-d_i\mu_{(i)}$ &(NR)\\ 
     \hline\hline
     \multirow{6}{*}{Category II: $z_{(i)}$}&\eqref{eq:sysODEreduced} &$z_{(i)}$ &$\dot z_{(i)}$ &(DN) \\  
     &\eqref{eq:z} &$z_{(i)}$ & $\dot z_{(i)}$ &(DN)\\
     &\eqref{eq:mu} &$\dot \mu_{(i)}-d_i\mu_{(i)}$ &$\ddot \mu_{(i)}-d_i\dot\mu_{(i)}$ &(RR)\\
     \cline{2-5} 
     &\eqref{eq:sysODEreduced} &$\dot z_{(i)}$ &$z_{(i)}$ &(ND)  \\  
     &\eqref{eq:z} &$\dot z_{(i)}$ &$z_{(i)}$ &(ND)\\
     &\eqref{eq:mu} &$\ddot \mu_{(i)}-d_i\dot\mu_{(i)}$ &$\dot \mu_{(i)}-d_i\mu_{(i)}$ &(RR)\\
     \hline\hline
     \multirow{6}{*}{Category III: $\mu_{(i)}$}&\eqref{eq:sysODEreduced} &$\mu_{(i)}$ &$\dot \mu_{(i)}$ &(DN) \\ 
     &\eqref{eq:z} &$\dot z_{(i)}+d_i z_{(i)}$ &$\ddot z_{(i)}+d_i \dot z_{(i)}$ &(RR)\\
     &\eqref{eq:mu} &$\mu_{(i)}$ &$\dot\mu_{(i)}$ &(DN)\\
     \cline{2-5} 
     &\eqref{eq:sysODEreduced}&$\dot\mu_{(i)}$ &$ \mu_{(i)}$ &(ND)  \\  
     &\eqref{eq:z} &$\ddot z_{(i)}+d_i \dot z_{(i)}$ &$\dot z_{(i)}+d_i z_{(i)}$ &(RR) \\
     &\eqref{eq:mu} &$\dot\mu_{(i)}$ &$\mu_{(i)}$ &(ND)
    \end{tabular}
  \end{center}
  \caption{Combinations of the DN and ND algorithms. The letter R stands for a Robin type condition.}
  \label{tab:combination}
\end{table}
The algorithms can be classified into three main categories, and
  each category is composed of two blocks, the first block represents
  a DN technique applied to~\eqref{eq:sysODEreduced}, whereas the
  second block represents a ND technique. Each block contains three
 rows: the first row is the algorithm applied to
formulation~\eqref{eq:sysODEreduced}, the second row the algorithm
applied to formulation~\eqref{eq:z} and the third row the algorithm
applied to formulation~\eqref{eq:mu}.

\begin{remark}\label{rem:transfer}
  In Table~\ref{tab:combination}, the transmission conditions $\ddot
  z_{(i)} + d_i \dot z_{(i)}$ and $\ddot \mu_{(i)} - d_i\dot
  \mu_{(i)}$ are in fact Robin type conditions, since, using the
  identity~\eqref{eq:2id} of $z_{(i)}$ and $\mu_{(i)}$, we find 
%  $\dotz_{(i)}=\ddot \mu_{(i)} - d_i\dot\mu_{(i)}$ and $\dot
%  \mu_{(i)}=\ddot z_{(i)} + d_i\dot z_{(i)}.$
   \[\dot z_{(i)}=\ddot \mu_{(i)} - d_i\dot\mu_{(i)}, \quad \dot \mu_{(i)}=\ddot z_{(i)} + d_i\dot z_{(i)}.\]
  On the other hand, from the first equation of~\eqref{eq:z} and of~\eqref{eq:mu}, we have
%  $\ddot z_{(i)} - \sigma_i^2 z_{(i)}=0$ and $\ddot \mu_{(i)} - \sigma_i^2 \mu_{(i)}=0$. 
   \[\ddot z_{(i)} - \sigma_i^2 z_{(i)}=0,\quad \ddot \mu_{(i)} - \sigma_i^2 \mu_{(i)}=0.\] 
  Substituting $\ddot z_{(i)}$ and $\ddot \mu_{(i)}$ gives 
%  $\dot \mu_{(i)}=\ddot z_{(i)} + d_i \dot z_{(i)} = d_i\dot z_{(i)} + \sigma_i^2 z_{(i)}$ and $\dot z_{(i)}=\ddot \mu_{(i)} - d_i\dot \mu_{(i)} = \sigma_i^2 \mu_{(i)}-d_i\dot\mu_{(i)}$.
   \[\dot \mu_{(i)}=\ddot z_{(i)} + d_i \dot z_{(i)} = d_i\dot z_{(i)} + \sigma_i^2 z_{(i)}, \quad \dot z_{(i)}=\ddot \mu_{(i)} - d_i\dot \mu_{(i)} = \sigma_i^2 \mu_{(i)}-d_i\dot\mu_{(i)}.\]
  Thus the transmission conditions containing a second derivative
  in Table~\ref{tab:combination} are indeed Robin type
  conditions. We decided to keep the notations $\ddot z_{(i)}$ and
  $\ddot \mu_{(i)}$ in Table~\ref{tab:combination} to show the direct
  link between the two states $z_{(i)}$ and $\mu_{(i)}$.

  However, there are other interpretations of some transmission
  conditions in certain circumstances. For instance, let us take the
  Neumann condition $\dot z_{(i)}$ in the second block of Category
  II for the problem~\eqref{eq:sysODEreduced}, it can also be
  interpreted as a Robin condition $\sigma_i^2
  \mu_{(i)}-d_i\dot\mu_{(i)}$ using the above argument. Then, this
  algorithm can also be read as a Robin-Dirichlet (RD) type algorithm
  instead of a Neumann-Dirichlet type. Moreover, this interpretation
  is particularly useful in this case, since it reveals the fact that
  the forward-backward property of the
  problem~\eqref{eq:sysODEreduced} is still kept by this
  algorithm. Otherwise, we can also use the identity of $\mu_{(i)}$
  in~\eqref{eq:2id} to transfer this Neumann condition $\dot z_{(i)}$
  to $\mu_{(i)}-d_i z_{(i)}$. This is also useful from a numerical
  point of view, since we can transfer a Neumann condition to a
  Dirichlet type condition. This will be used in detail in the
  following analysis.
\end{remark}

\subsection{Category I}\label{sec:3.1}

We start with the algorithms in Category I, which run on the
pair $(z_{(i)},\mu_{(i)})$ to solve~\eqref{eq:sysODEreduced}, and study
the DN and then the ND variant.

\subsubsection{Dirichlet-Neumann algorithm \texorpdfstring{(DN$_1$)}{DN1}} 

This is~\eqref{eq:DN1}, at first sight the most natural method
  that keeps the forward-backward structure as in the original
problem~\eqref{eq:sysODEreduced}. To analyze the convergence
behavior, we can choose any of the problem
formulations~\eqref{eq:z}, \eqref{eq:mu}, since they are equivalent
to~\eqref{eq:sysODEreduced}. Choosing~\eqref{eq:z}, the algorithm 
DN$_1$ for $i=1,\ldots,n,$ and iteration $k=1,2,\ldots$ is given by
\begin{equation}\label{eq:errDN1}
  \left\{
    \begin{aligned}
      \ddot z_{1,(i)}^k - \sigma_i^2 z_{1,(i)}^k &= 0 \text{ in } \Omega_1,\\
      z_{1,(i)}^k(0) &= 0, \\
      \dot z_{1,(i)}^k(\alpha) + d_i z_{1,(i)}^k(\alpha) &= f_{\alpha,(i)}^{k-1},
    \end{aligned}
  \right.
  \quad
  \left\{
    \begin{aligned}
      \ddot z_{2,(i)}^k - \sigma_i^2 z_{2,(i)}^k &= 0 \text{ in } \Omega_2,\\
      \dot z_{2,(i)}^k(\alpha) &= \dot z_{1,(i)}^k(\alpha),\\
      \dot z_{2,(i)}^k(T) + \omega_i z_{2,(i)}^k(T) &= 0,
    \end{aligned}
  \right.
\end{equation}
and the update of the transmission condition defined
in~\eqref{eq:DN1tran} becomes
\begin{equation}\label{eq:errDN1tran}
  f_{\alpha,(i)}^{k}= (1-\theta)f_{\alpha,(i)}^{k-1} + \theta\big(\dot z_{2,(i)}^k(\alpha) + d_i z_{2,(i)}^k(\alpha)\big). 
\end{equation}
This is a Robin-Neumann type algorithm applied to solve the
problem~\eqref{eq:z}. Using the general
solution~\eqref{eq:gensol}, and the initial and final condition, we
find
\begin{equation}\label{eq:zsol}
  z_{1,(i)}^k(t) = A_i^k \sinh(\sigma_i t), \, z_{2,(i)}^k(t) =  B_i^k \Big( \sigma_i\cosh\big(\sigma_i (T-t)\big) + \omega_i\sinh\big(\sigma_i (T-t)\big) \Big),
\end{equation}
where $A_i^k$ and $B_i^k$ are determined by the transmission conditions at $\alpha$ in~\eqref{eq:errDN1}. Note that we will use~\eqref{eq:zsol} in the analysis for all algorithms, since only the transmission conditions will change. Inserting~\eqref{eq:zsol} at the interface $\alpha$ into~\eqref{eq:errDN1} and solving for $A_i^k$, $B_i^k$ gives 
%$A_i^k =\frac{f_{\alpha,(i)}^{k-1}}{\sigma_i\cosh(a_i) + d_i\sinh(a_i)}$ and $B_i^k =\frac{-f_{\alpha,(i)}^{k-1}\cosh(a_i)}{(\sigma_i\cosh(a_i) +d_i\sinh(a_i))(\sigma_i\sinh(b_i) + \omega_i\cosh(b_i))}$,
 \[\begin{aligned}
   A_i^k &= \frac{f_{\alpha,(i)}^{k-1}}{\sigma_i\cosh(a_i) + d_i\sinh(a_i)},\\
   B_i^k &= -\frac{f_{\alpha,(i)}^{k-1}\cosh(a_i)}{\Big(\sigma_i\cosh(a_i) + d_i\sinh(a_i)\Big)\Big(\sigma_i\sinh(b_i) + \omega_i\cosh(b_i) \Big)}, 
 \end{aligned}\]
where we let $a_i:=\sigma_i\alpha$ and $b_i:=\sigma_i (T-\alpha)$ to simplify the notations, and $a_i+b_i = \sigma_i T$. Using the update of the transmission condition~\eqref{eq:errDN1tran}, we obtain
 \[\begin{aligned}
   f_{\alpha,(i)}^{k}  = (1-\theta)f_{\alpha,(i)}^{k-1} + &\theta(\dot z_{2,(i)}^k(\alpha) + d_i z_{2,(i)}^k(\alpha))\\
   =(1-\theta)f_{\alpha,(i)}^{k-1} + &\theta B_i^k\Big((\omega_id_i-\sigma_i^2)\sinh(b_i)+\sigma_i(d_i-\omega_i)\cosh(b_i)\Big)\\
   =(1-\theta)f_{\alpha,(i)}^{k-1} + &\theta f_{\alpha,(i)}^{k-1}\nu^{-1}\frac{\sigma_i\gamma+\beta_i\tanh(b_i)}{\big(\sigma_i + d_i\tanh(a_i)\big)\big(\omega_i + \sigma_i\tanh(b_i)\big)}.
 \end{aligned}\]
%$f_{\alpha,(i)}^{k} =(1-\theta)f_{\alpha,(i)}^{k-1} + \theta f_{\alpha,(i)}^{k-1}\nu^{-1}\frac{\sigma_i\gamma+\beta_i\tanh(b_i)}{(\sigma_i + d_i\tanh(a_i))(\omega_i + \sigma_i\tanh(b_i))}$,
which leads to the following result.

\begin{theorem}
  The algorithm DN$_1$~\eqref{eq:DN1}-\eqref{eq:DN1tran} converges if and only if
  \begin{equation}\label{eq:rhoDN1}
    \rho_{\text{DN}_1}:=\max_{d_i\in\lambda(A)}\Big|1-\theta\Big(1-\nu^{-1}\frac{\sigma_i\gamma+\beta_i\tanh(b_i)}{\big(\sigma_i + d_i\tanh(a_i)\big)\big(\omega_i + \sigma_i\tanh(b_i)\big)}\Big)\Big|<1,
  \end{equation}
  where $\lambda(A)$ is the spectrum of the matrix $A$.
\end{theorem}
\begin{remark}
  Instead of focusing on the state $z_{(i)}$ for the analysis, we could also have focused on the state $\mu_{(i)}$, which gives the same result, see Appendix~\ref{sec:app1}.
\end{remark}
To get more insight in the convergence behavior, we consider a few special cases.
\begin{corollary}\label{cor:cvDN1cla}
If the matrix $A$ is not singular, then the algorithm DN$_1$~\eqref{eq:DN1}-\eqref{eq:DN1tran} for $\theta=1$ converges for all initial guesses.
\end{corollary}
\begin{proof}
Substituting $\theta=1$ into~\eqref{eq:rhoDN1}, we have
  \begin{equation}\label{eq:rhoDN_1}
    \rho_{\text{DN}_1}|_{\theta=1}=\nu^{-1}\max_{d_i\in\lambda(A)}\Big|\frac{\sigma_i\gamma+\beta_i\tanh(b_i)}{\big(\sigma_i + d_i\tanh(a_i)\big)\big(\omega_i + \sigma_i\tanh(b_i)\big)}\Big|.
  \end{equation}
  Using the definition of $\sigma_i, \beta_i$ and $\omega_i$ from~\eqref{eq:sob}, the numerator can be written as
%  $\sigma_i\gamma+\beta_i\tanh(b_i) = \gamma(\sigma_i - d_i\tanh(b_i)) + \tanh(b_i)$.
   \[\sigma_i\gamma+\beta_i\tanh(b_i) = \gamma\big(\sigma_i - d_i\tanh(b_i)\big) + \tanh(b_i).\]
  Since $0<\tanh(x)<1$, $\forall x>0$ and $\sigma_i - d_i\tanh(b_i) >0$, both the numerator and the denominator in~\eqref{eq:rhoDN_1} are positive. Now the difference between the numerator and the denominator is
   \[\begin{aligned}
    &\big(\sigma_i + d_i\tanh(a_i)\big)\big(\omega_i + \sigma_i\tanh(b_i)\big)-\nu^{-1}\big(\sigma_i\gamma+\beta_i\tanh(b_i)\big)\\
     =&\big(\sigma_i + d_i\tanh(a_i)\big)\big(\omega_i + \sigma_i\tanh(b_i)\big)-\big(\sigma_i(\omega_i-d_i)+(\sigma_i^2-\omega_i d_i)\tanh(b_i)\big)\\
     =&\omega_id_i\big(\tanh(b_i)+\tanh(a_i)\big)+\sigma_id_i\big(1+\tanh(b_i)\tanh(a_i)\big)\\
     =&\big(1+\tanh(b_i)\tanh(a_i)\big)\big(\sigma_id_i + \omega_id_i\tanh(\sigma_i T)\big) >0,
   \end{aligned}\]
%  $(\sigma_i + d_i\tanh(a_i))(\omega_i + \sigma_i\tanh(b_i))-\nu^{-1}(\sigma_i\gamma+\beta_i\tanh(b_i))=(1+\tanh(b_i)\tanh(a_i))(\sigma_id_i + \omega_id_i\tanh(\sigma_i T)) >0$,
  meaning that for each eigenvalue $d_i$, 
%  $0<\nu^{-1}\frac{\sigma_i\gamma+\beta_i\tanh(b_i)}{(\sigma_i + d_i\tanh(a_i))(\omega_i + \sigma_i\tanh(b_i))}<1$.
   \[0<\nu^{-1}\frac{\sigma_i\gamma+\beta_i\tanh(b_i)}{\big(\sigma_i + d_i\tanh(a_i)\big)\big(\omega_i + \sigma_i\tanh(b_i)\big)}<1.\]
   This concludes the proof.
\end{proof}

\begin{remark}\label{rem:DN10}
  For the Laplace operator with homogeneous Dirichlet boundary
  conditions in our model problem~\eqref{eq:heat}, there is no zero
  eigenvalue for its discretization matrix $A$. If an eigenvalue
  $d_i=0$, we have $\sigma_i|_{d_i=0}=\sqrt{\nu^{-1}}$,
  $\omega_i|_{d_i=0}=\gamma\nu^{-1}$ and
  $\beta_i|_{d_i=0}=1$. Substituting these values into the convergence
  factor~\eqref{eq:rhoDN_1}, we find
   \[\rho_{\text{DN}_1}|_{\theta=1,d_i=0}=\nu^{-1}\frac{\sqrt{\nu^{-1}}\gamma+\tanh\left(\sqrt{\nu^{-1}} (T-\alpha)\right)}{\sqrt{\nu^{-1}} \big(\gamma\nu^{-1} + \sqrt{\nu^{-1}}\tanh\left(\sqrt{\nu^{-1}} (T-\alpha)\right)\big)} = 1,\]
%  $\rho_{\text{DN}_1}|_{\theta=1,d_i=0}=\nu^{-1}\frac{\sqrt{\nu^{-1}}\gamma+\tanh(\sqrt{\nu^{-1}} (T-\alpha))}{\sqrt{\nu^{-1}} (\gamma\nu^{-1} + \sqrt{\nu^{-1}}\tanh(\sqrt{\nu^{-1}} (T-\alpha)))} = 1$,
and convergence is lost. The convergence behavior of the
  algorithm DN$_1$ for small eigenvalues is thus not
good. Furthermore, inserting $d_i=0$ into~\eqref{eq:rhoDN1} and
using the above result, we find that $\rho_{\text{DN}_1}|_{d_i=0}=1$,
independently of the relaxation parameter $\theta$ and the interface
position $\alpha$: relaxation can not fix this problem.
\end{remark}

\begin{remark}\label{rem:DN1lim}
  If some $d_i$ goes to infinity, we have $\sigma_i \sim_{\infty} d_i$ and $\omega_i\sim_{\infty} d_i$, and therefore
   \[\lim_{d_i\rightarrow\infty}\Big|1-\theta\Big(1-\nu^{-1}\frac{\sigma_i\gamma+\beta_i\tanh(b_i)}{\big(\sigma_i + d_i\tanh(a_i)\big)\big(\omega_i + \sigma_i\tanh(b_i)\big)}\Big)\Big| = |1-\theta|,\]
%  $\lim_{d_i\rightarrow\infty}\big|1-\theta\big(1-\nu^{-1}\frac{\sigma_i\gamma+\beta_i\tanh(b_i)}{(\sigma_i + d_i\tanh(a_i))(\omega_i + \sigma_i\tanh(b_i))}\big)\big| = |1-\theta|$,
  which is independent of $\alpha$, so high frequency convergence is robust with relaxation. One can use $\theta=1$ to get a good smoother, with the following convergence factor estimate.
\end{remark}

%For DN$_1$, equioscillation of the convergence factor between zero eigenvalue (i.e., $\rho_{\text{DN}_1}|_{d_i=0}$) and large eigenvalues (i.e., $\rho_{\text{DN}_1}|_{d_i\rightarrow\infty}$) cannot provide the best convergence factor $\theta_{\text{DN}_1}^{\star}$, since the convergence behavior for zero eigenvalue is independent of the choice of the relaxation parameter $\theta$. One can use $\theta=1$ to get a good smoother, with the following convergence factor estimate:
  
\begin{corollary}\label{cor:cvDN1estim}
  If $A$ is positive semi-definite, then the algorithm
  DN$_1$~\eqref{eq:DN1}-\eqref{eq:DN1tran} with $\theta=1$ satisfies
  the convergence estimate $\rho_{\text{DN}_1}|_{\theta=1} \leq
  \frac{1+\gamma\sigma_{\min}}{\nu d_{\min}^2}$, with
  $d_{\min}:=\min\lambda(A)$ the smallest eigenvalue of $A$.
\end{corollary}
\begin{proof}
  Since for $\theta=1$, Corollary~\ref{cor:cvDN1cla} shows that the
  convergence factor is between 0 and 1 for each
  eigenvalue $d_i$, we can take~\eqref{eq:rhoDN_1} and remove the
  absolute value,
  \[\rho_{\text{DN}_1}|_{\theta=1}=\nu^{-1}\max_{d_i\in\lambda(A)}\frac{\tanh(b_i)+\gamma\big(\sigma_i- d_i\tanh(b_i)\big)}{\big(\sigma_i + d_i\tanh(a_i)\big)\big(\omega_i + \sigma_i\tanh(b_i)\big)}.\]
%  $\rho_{\text{DN}_1}|_{\theta=1}=\nu^{-1}\max_{d_i\in\lambda(A)}\frac{\tanh(b_i)+\gamma(\sigma_i- d_i\tanh(b_i))}{(\sigma_i + d_i\tanh(a_i))(\omega_i + \sigma_i\tanh(b_i))}$.
  Using the definition of $\sigma_i$ and $\omega_i$ from~\eqref{eq:sob}, we have $\sigma_i > d_i\geq 0$ and $\omega_i\geq d_i\geq 0$. Since $0<\tanh(x)<1$, $\forall x>0$, we obtain that $\sigma_i + d_i\tanh(a_i) \geq d_i$, $\omega_i + \sigma_i\tanh(b_i) \geq d_i$ and $\sigma_i- d_i\tanh(b_i) \leq \sigma_i$. This implies
   \[\frac{\tanh(b_i)+\gamma\big(\sigma_i- d_i\tanh(b_i)\big)}{\big(\sigma_i + d_i\tanh(a_i)\big)\big(\omega_i + \sigma_i\tanh(b_i)\big)}\leq\frac{1+\gamma\sigma_i}{d_i^2} = \frac1{d_i}(\frac1{d_i}+\gamma\frac{\sigma_i}{d_i}).\]
%  $\frac{\tanh(b_i)+\gamma(\sigma_i- d_i\tanh(b_i))}{(\sigma_i + d_i\tanh(a_i))(\omega_i + \sigma_i\tanh(b_i))}\leq\frac{1+\gamma\sigma_i}{d_i^2} = \frac1{d_i}(\frac1{d_i}+\gamma\frac{\sigma_i}{d_i})$.
  Using once again the definition of $\sigma_i$ from~\eqref{eq:sob}, we find
   \[\frac{\sigma_i}{d_i} = \sqrt{1+\frac{\nu^{-1}}{d_i^2}}\leq\sqrt{1+\frac{\nu^{-1}}{d_{\min}^2}}.\]
%  $\frac{\sigma_i}{d_i} = \sqrt{1+\frac{\nu^{-1}}{d_i^2}}\leq\sqrt{1+\frac{\nu^{-1}}{d_{\min}^2}}$.
  Hence, we have 
   \[\frac{1+\gamma\sigma_i}{d_i^2}\leq \frac{1+\gamma\sigma_{\min}}{d_{\min}^2},\]
%  $\frac{1+\gamma\sigma_i}{d_i^2}\leq \frac{1+\gamma\sigma_{\min}}{d_{\min}^2}$,
  which concludes the proof.
\end{proof}
Since $A$ comes from a spatial discretization, the smallest
eigenvalue of $A$ depends only little on the spatial mesh size,
  and convergence is thus robust under mesh refinement. Corollary
  \ref{cor:cvDN1estim} is however less useful when $\nu$ is small: for
  example for $\gamma=0$, the bound is less than one only if
  $\nu>\frac1{d_{\min}^2}$, but we have also the following
  convergence result.
\begin{theorem}
  The algorithm DN$_1$~\eqref{eq:DN1}-\eqref{eq:DN1tran} converges for
  all initial guesses under the assumption that the matrix $A$ is not
  singular.
\end{theorem}
\begin{proof}
From Corollary~\ref{cor:cvDN1cla}, we know that the convergence factor
satisfies $0<\rho_{\text{DN}_1}|_{\theta=1}<1$. Using its definition
\eqref{eq:rhoDN1}, we find for $\theta\in(0,1)$,
 \[0<1-\theta<\rho_{\text{DN}_1} = 1 - \theta (1-\rho_{\text{DN}_1}|_{\theta=1})<1.\]
%$0<1-\theta<\rho_{\text{DN}_1} = 1 - \theta (1-\rho_{\text{DN}_1}|_{\theta=1})<1$,
which concludes the proof.
\end{proof}

\begin{remark}
As shown in the previous proof, the function $g(\theta):=1 - \theta (1-\rho_{\text{DN}_1}|_{\theta=1})$ is decreasing for $\theta\in(0,1)$, which makes $\theta=1$ the best relaxation parameter. This is further confirmed by our numerical experiments (see Figure~\ref{fig:cv1_5th}). Due to the bad convergence behavior of the algorithm DN$_1$ for small eigenvalues, it only makes this most natural DN algorithm a good smoother but not a good solver.
\end{remark}

\subsubsection{Neumann-Dirichlet algorithm \texorpdfstring{(ND$_1$)}{ND1}}

We now invert the two conditions, and apply the Neumann condition
to the state $\mu_{(i)}$ in $\Omega_1$ and the Dirichlet condition to
the state $z_{(i)}$ in $\Omega_2$, still respecting the forward-backward structure. For iteration index $k=1,2,\ldots$, the algorithm ND$_1$ computes
\begin{equation}\label{eq:ND1}
  \begin{aligned}
    &\left\{
      \begin{aligned}
        \begin{pmatrix}
          \dot z_{1,(i)}^k\\
          \dot \mu_{1,(i)}^k
        \end{pmatrix}
        +
        \begin{pmatrix}
          d_i & -\nu^{-1} \\
          -1 & -d_i
        \end{pmatrix}
        \begin{pmatrix}
          z_{1,(i)}^k\\
          \mu_{1,(i)}^k
        \end{pmatrix}
        &=
        \begin{pmatrix}
          0\\
          0
        \end{pmatrix} \text{ in } \Omega_1,\\
        z_{1,(i)}^k(0) &= 0,\\
        \dot\mu_{1,(i)}^k(\alpha) & = \dot\mu_{2,(i)}^k(\alpha),
      \end{aligned}
    \right.\\
    &\left\{
      \begin{aligned}
        \begin{pmatrix}
          \dot z_{2,(i)}^k\\
          \dot \mu_{2,(i)}^k
        \end{pmatrix}
        +
        \begin{pmatrix}
          d_i & -\nu^{-1} \\
          -1 & -d_i
        \end{pmatrix}
        \begin{pmatrix}
          z_{2,(i)}^k\\
          \mu_{2,(i)}^k
        \end{pmatrix}
        &=
        \begin{pmatrix}
          0\\
          0
        \end{pmatrix} \text{ in } \Omega_2,\\
        z_{2,(i)}^k(\alpha) &= f_{\alpha,(i)}^{k-1},\\
        \mu_{2,(i)}^k(T) + \gamma z_{2,(i)}^k(T)&=0,
      \end{aligned}
    \right.
  \end{aligned}
\end{equation}
and we update the transmission condition by 
\begin{equation}\label{eq:ND1tran}
  f_{\alpha,(i)}^{k}:= (1-\theta)f_{\alpha,(i)}^{k-1} + \theta z_{1,(i)}^k(\alpha), \quad\theta\in(0,1).
\end{equation}
For the convergence analysis, we choose to use the
  formulation~\eqref{eq:mu}, i.e.
\begin{equation}\label{eq:errND1}
  \left\{
    \begin{aligned}
      \ddot \mu_{1,(i)}^k - \sigma_i^2 \mu_{1,(i)}^k &= 0 \text{ in } \Omega_1,\\
      \dot \mu_{(i)}(0)-d_i\mu_{(i)}(0) &= 0, \\
      \dot \mu_{1,(i)}^k(\alpha) &= \dot \mu_{2,(i)}^k(\alpha),
    \end{aligned}
  \right.
  \quad
  \left\{
    \begin{aligned}
      \ddot \mu_{2,(i)}^k - \sigma_i^2 \mu_{2,(i)}^k &= 0 \text{ in } \Omega_2,\\
      \dot \mu_{2,(i)}^k(\alpha) - d_i\mu_{2,(i)}^k(\alpha) &= f_{\alpha,(i)}^{k-1},\\
      \gamma \dot \mu_{(i)}(T)+\beta_i\mu_{(i)}(T)&=0,
    \end{aligned}
  \right.
\end{equation}
where the update of the transmission condition~\eqref{eq:ND1tran} becomes 
\begin{equation}\label{eq:errND1tran}
  f_{\alpha,(i)}^{k}= (1-\theta)f_{\alpha,(i)}^{k-1} + \theta \big(\dot \mu_{1,(i)}^k(\alpha) - d_i\mu_{1,(i)}^k(\alpha)\big), \quad\theta\in(0,1). 
\end{equation}
The algorithm ND$_1$~\eqref{eq:ND1} can thus be interpreted as
    a NR type algorithm~\eqref{eq:errND1}. 
  Using the general solution~\eqref{eq:gensol} and the initial
  and final conditions, we get
\begin{equation}\label{eq:musol}
  \begin{aligned}
    \mu_{1,(i)}^k(t) &= A_i^k \big(\sigma_i\cosh(\sigma_i t) + d_i\sinh(\sigma_i t)\big),\\
    \mu_{2,(i)}^k(t) &= B_i^k \Big(\gamma\sigma_i\cosh\big(\sigma_i (T-t)\big) + \beta_i\sinh\big(\sigma_i (T-t)\big) \Big),
  \end{aligned}
\end{equation}
and from the transmission condition in~\eqref{eq:errND1} on each
domain, and we obtain
 \[\begin{aligned}
   A_i^k&= \frac{f_{\alpha,(i)}^{k-1}\big(\sigma_i\gamma\sinh(b_i)+\beta_i\cosh(b_i)\big)}{\big(\omega_i\sinh(b_i) +\sigma_i\cosh(b_i)\big)\big(\sigma_i\sinh(a_i) + d_i\cosh(a_i)\big)},\\ B_i^k&=\frac{-f_{\alpha,(i)}^{k-1}}{\omega_i\sinh(b_i) +\sigma_i\cosh(b_i) }.
 \end{aligned}\]
%$A_i^k= \frac{f_{\alpha,(i)}^{k-1}(\sigma_i\gamma\sinh(b_i)+\beta_i\cosh(b_i))}{(\omega_i\sinh(b_i) +\sigma_i\cosh(b_i))(\sigma_i\sinh(a_i) + d_i\cosh(a_i))}$ and $B_i^k=\frac{-f_{\alpha,(i)}^{k-1}}{\omega_i\sinh(b_i) +\sigma_i\cosh(b_i) }$.
Using the relation~\eqref{eq:errND1tran}, we find
 \[\begin{aligned}
   f_{\alpha,(i)}^{k} &= (1-\theta)f_{\alpha,(i)}^{k-1} +\theta(\dot \mu_{1,(i)}^k(\alpha) - d_i\mu_{1,(i)}^k(\alpha))\\
   &=(1-\theta)f_{\alpha,(i)}^{k-1} +\theta A_i^k\nu^{-1}\sinh(a_i)\\
   &=(1-\theta)f_{\alpha,(i)}^{k-1} + \theta f_{\alpha,(i)}^{k-1}\nu^{-1}\frac{\sigma_i\gamma + \beta_i\coth(b_i)}{\big(\sigma_i + d_i\coth(a_i)\big)\big(\omega_i+\sigma_i\coth(b_i)\big)}.
 \end{aligned}\]
%$f_{\alpha,(i)}^{k}=(1-\theta)f_{\alpha,(i)}^{k-1} + \theta f_{\alpha,(i)}^{k-1}\nu^{-1}\frac{\sigma_i\gamma + \beta_i\coth(b_i)}{(\sigma_i + d_i\coth(a_i))(\omega_i+\sigma_i\coth(b_i))}$,
which leads to the following result.

\begin{theorem}
  The algorithm ND$_1$~\eqref{eq:ND1}-\eqref{eq:ND1tran} converges if and only if
  \begin{equation}\label{eq:rhoND1}
    \rho_{\text{ND}_1}:=\max_{d_i\in\lambda(A)}\Big|1-\theta\Big(1-\nu^{-1}\frac{\sigma_i\gamma + \beta_i\coth(b_i)}{\big(\sigma_i + d_i\coth(a_i)\big)\big(\omega_i+\sigma_i\coth(b_i)\big)}\Big)\Big|<1.
  \end{equation}
\end{theorem}
The convergence factor of the algorithm ND$_1$~\eqref{eq:rhoND1} is
very similar to that of DN$_1$~\eqref{eq:rhoDN1}. For instance, the
behavior for large and small eigenvalues shown in
Remarks~\ref{rem:DN10} and \ref{rem:DN1lim} still hold:
when inserting $d_i=0$ into~\eqref{eq:rhoND1} we find
 \[\rho_{\text{ND}_1}|_{d_i=0}=\Big|1-\theta\Big(1-\nu^{-1}\frac{\sqrt{\nu^{-1}}\gamma+\coth\left(\sqrt{\nu^{-1}} (T-\alpha)\right)}{\sqrt{\nu^{-1}} \big(\gamma\nu^{-1} + \sqrt{\nu^{-1}}\coth\left(\sqrt{\nu^{-1}} (T-\alpha)\right)\big)}\Big)\Big| = 1,\]
%$\rho_{\text{ND}_1}|_{d_i=0}=|1-\theta(1-\nu^{-1}\frac{\sqrt{\nu^{-1}}\gamma+\coth(\sqrt{\nu^{-1}} (T-\alpha))}{\sqrt{\nu^{-1}} (\gamma\nu^{-1} + \sqrt{\nu^{-1}}\coth(\sqrt{\nu^{-1}} (T-\alpha)))})| = 1$,
again independent of the relaxation parameter $\theta$ and the interface position $\alpha$; and when the eigenvalue $d_i$ goes to infinity, we find
 \[\lim_{d_i\rightarrow\infty}\Big|1-\theta\Big(1-\nu^{-1}\frac{\sigma_i\gamma+\beta_i\coth(b_i)}{\big(\sigma_i + d_i\coth(a_i)\big)\big(\omega_i + \sigma_i\coth(b_i)\big)}\Big)\Big| = |1-\theta|,\]
%$\lim_{d_i\rightarrow\infty}|1-\theta(1-\nu^{-1}\frac{\sigma_i\gamma+\beta_i\coth(b_i)}{(\sigma_i + d_i\coth(a_i))(\omega_i + \sigma_i\coth(b_i))})| = |1-\theta|$,
again independent of the interface position $\alpha$.
Due however to the presence of the hyperbolic cotangent function
in~\eqref{eq:rhoND1} instead of the hyperbolic tangent function
in~\eqref{eq:rhoDN1}, we need further assumptions to obtain
results like Corollaries~\ref{cor:cvDN1cla} and
\ref{cor:cvDN1estim}. Indeed, substituting $\theta=1$
into~\eqref{eq:rhoND1} and using the definition of $\sigma_i$, $\beta_i$
from~\eqref{eq:sob}, the numerator reads
 \[\begin{aligned}
   \sigma_i\gamma + \beta_i\coth(b_i)=\gamma\Big(\sqrt{d_i^2+\nu^{-1}}-d_i\coth\big(\sqrt{d_i^2+\nu^{-1}}&(T-\alpha)\big)\Big) \\
   + &\coth\big(\sqrt{d_i^2+\nu^{-1}}(T-\alpha)\big).
 \end{aligned}\]
%$\sigma_i\gamma + \beta_i\coth(b_i)=\gamma(\sqrt{d_i^2+\nu^{-1}}-d_i\coth(\sqrt{d_i^2+\nu^{-1}}(T-\alpha))) + \coth(\sqrt{d_i^2+\nu^{-1}}(T-\alpha))$.
Depending on $\gamma,\nu$ and $\alpha$, this value could be negative. However, by setting $\gamma=0$, the numerator is guaranteed to be positive, and we obtain the following results.

\begin{corollary}\label{cor:cvND1cla}
  If $A$ is not singular and the parameter $\gamma=0$, then
  the algorithm ND$_1$~\eqref{eq:ND1}-\eqref{eq:ND1tran} for
  $\theta=1$ converges for all initial guesses.
\end{corollary}

\begin{proof}
  Substituting $\theta=1$ and $\gamma=0$ into~\eqref{eq:rhoND1}, we get
  \begin{equation}\label{eq:rhoND_1}
    \rho_{\text{ND}_2}|_{\theta=1}=\nu^{-1}\max_{d_i\in\lambda(A)}\Big|\frac{\coth(b_i)}{\big(\sigma_i + d_i\coth(a_i)\big)\big(d_i + \sigma_i\coth(b_i)\big)}\Big|.
  \end{equation}
  Since $\coth(x)>1$, $\forall x>0$, both the numerator and the denominator in~\eqref{eq:rhoND_1} are positive, and the difference between them is
  \[\begin{aligned}
     &\big(\sigma_i + d_i\coth(a_i)\big)\big(d_i + \sigma_i\coth(b_i)\big)-\nu^{-1}\coth(b_i)\\
     =&\big(\sigma_i + d_i\coth(a_i)\big)\big(d_i + \sigma_i\coth(b_i)\big)-(\sigma_i^2- d_i^2)\coth(b_i)\\
     =&d_i^2\big(\coth(b_i)+\coth(a_i)\big)+\sigma_id_i\big(1+\coth(b_i)\coth(a_i)\big)\\
    =&\big(\coth(a_i)+\coth(b_i)\big)\big(d_i^2 + \sigma_id_i\coth(\sigma_i T)\big) >0,
   \end{aligned}\]
%  $(\sigma_i + d_i\coth(a_i))(d_i + \sigma_i\coth(b_i))-\nu^{-1}\coth(b_i)=(\coth(a_i)+\coth(b_i))(d_i^2 + \sigma_id_i\coth(\sigma_i T)) >0$,
  meaning that for each eigenvalue $d_i$, 
   \[0<\nu^{-1}\frac{\coth(b_i)}{\big(\sigma_i + d_i\coth(a_i)\big)\big(\omega_i + \sigma_i\coth(b_i)\big)}<1,\]
%  $0<\nu^{-1}\frac{\coth(b_i)}{(\sigma_i + d_i\coth(a_i))(\omega_i + \sigma_i\coth(b_i))}<1$,
  which concludes the proof.
\end{proof}

\begin{corollary}
  If $A$ is positive semi-definite and the parameter $\gamma=0$,
  then the algorithm ND$_1$~\eqref{eq:ND1}-\eqref{eq:ND1tran} for
  $\theta=1$ satisfies the convergence estimate
  \begin{equation}\label{eq:ND1upbound}
    \rho_{\text{ND}_1}|_{\theta=1} \leq \frac{\coth\big(\sigma_{\min}(T-\alpha)\big)}{\nu(\sigma_{\min}+d_{\min})^2}.
  \end{equation}
\end{corollary}

\begin{proof}
  Since we have shown in Corollary~\ref{cor:cvND1cla} that the
  convergence factor is between 0 and 1 for each eigenvalue $d_i$, we
  can take~\eqref{eq:rhoND_1} and remove the absolute value,
  \[\rho_{\text{ND}_2}|_{\theta=1}=\nu^{-1}\max_{d_i\in\lambda(A)}\frac{\coth(b_i)}{\big(\sigma_i + d_i\coth(a_i)\big)\big(d_i + \sigma_i\coth(b_i)\big)}.\]
%  $\rho_{\text{ND}_2}|_{\theta=1}=\nu^{-1}\max_{d_i\in\lambda(A)}\frac{\coth(b_i)}{(\sigma_i + d_i\coth(a_i))(d_i + \sigma_i\coth(b_i))}$.
  Since $\sigma_i = \sqrt{d_i^2+\nu^{-1}} \geq d_i\geq 0$ and $\coth(x)>1$, $\forall x>0$, we obtain that $\sigma_i + d_i\coth(a_i) \geq \sigma_i+d_i$ and $d_i + \sigma_i\coth(b_i) \geq \sigma_i+d_i$. This implies that
  \[\frac{\coth(b_i)}{\big(\sigma_i + d_i\coth(a_i)\big)\big(d_i + \sigma_i\coth(b_i)\big)}\leq\frac{\coth(b_i)}{(\sigma_i+d_i)^2}.\]
%  $\frac{\coth(b_i)}{(\sigma_i + d_i\coth(a_i))(d_i + \sigma_i\coth(b_i))}\leq\frac{\coth(b_i)}{(\sigma_i+d_i)^2}$.
  Recalling $\coth(b_i)=\coth(\sigma_i(T-\alpha))$, and using the fact that $d_i\geq d_{\min}$ and $\sigma_i\geq \sigma_{\min}:=\sqrt{d_{\min}^2+\nu^{-1}}$, we find 
  \[\frac{\coth(b_i)}{(\sigma_i+d_i)^2} \leq \frac{\coth(\sigma_{\min}(T-\alpha))}{(\sigma_{\min}+d_{\min})^2},\]
%  $\frac{\coth(b_i)}{(\sigma_i+d_i)^2} \leq \frac{\coth(\sigma_{\min}(T-\alpha))}{(\sigma_{\min}+d_{\min})^2}$,
  which concludes the proof.
\end{proof}
Like for DN$_1$, the estimate~\eqref{eq:ND1upbound} is
independent of the spatial mesh size, and since for
  $\gamma=0$, the convergence factor satisfies
$0<\rho_{\text{ND}_1}|_{\theta=1}<1$ as shown in
Corollary~\ref{cor:cvND1cla}, using the definition of the convergence
factor~\eqref{eq:rhoND1}, we obtain the following result.
\begin{theorem}
  The algorithm ND$_1$~\eqref{eq:ND1}-\eqref{eq:ND1tran} converges for
  all initial guesses if $\gamma=0$ and the matrix $A$ is not
  singular.
\end{theorem}

\subsection{Category II}
We now study algorithms in Category II which run only on
the state $z_{(i)}$ to solve the problem~\eqref{eq:sysODEreduced}, based on
DN and ND techniques.

\subsubsection{Dirichlet-Neumann algorithm \texorpdfstring{(DN$_2$)}{DN2}}
As explained in Table~\ref{tab:combination}, we apply the Dirichlet condition in $\Omega_1$ and the Neumann condition in $\Omega_2$ both on the primal state $z_{(i)}$. For the iteration index $k=1,2,\ldots,$ the algorithm DN$_2$ solves
\begin{equation}\label{eq:DN2}
  \begin{aligned}
    &\left\{
      \begin{aligned}
        \begin{pmatrix}
          \dot z_{1,(i)}^k\\
          \dot \mu_{1,(i)}^k
        \end{pmatrix}
        +
        \begin{pmatrix}
          d_i & -\nu^{-1} \\
          -1 & -d_i
        \end{pmatrix}
        \begin{pmatrix}
          z_{1,(i)}^k\\
          \mu_{1,(i)}^k
        \end{pmatrix}
        &=
        \begin{pmatrix}
          0\\
          0
        \end{pmatrix} \text{ in } \Omega_1,\\
        z_{1,(i)}^k(0) &= 0,\\
        z_{1,(i)}^k(\alpha) & = f_{\alpha,(i)}^{k-1},
      \end{aligned}
    \right.\\
    &\left\{
      \begin{aligned}
        \begin{pmatrix}
          \dot z_{2,(i)}^k\\
          \dot \mu_{2,(i)}^k
        \end{pmatrix}
        +
        \begin{pmatrix}
          d_i & -\nu^{-1} \\
          -1 & -d_i
        \end{pmatrix}
        \begin{pmatrix}
          z_{2,(i)}^k\\
          \mu_{2,(i)}^k
        \end{pmatrix}
        &=
        \begin{pmatrix}
          0\\
          0
        \end{pmatrix} \text{ in } \Omega_2,\\
        \dot z_{2,(i)}^k(\alpha) &= \dot z_{1,(i)}^k(\alpha),\\
        \mu_{2,(i)}^k(T) + \gamma z_{2,(i)}^k(T)&=0,
      \end{aligned}
    \right.
  \end{aligned}
\end{equation}
and we update the transmission condition by 
\begin{equation}\label{eq:DN2tran}
  f_{\alpha,(i)}^{k}:= (1-\theta)f_{\alpha,(i)}^{k-1} + \theta z_{2,(i)}^k(\alpha), \quad \theta\in(0,1).
\end{equation}
At first glance, this algorithm does not have the
forward-backward structure, with both an initial and a
final condition on $z_{1,(i)}$ in $\Omega_1$ and nothing on
$\mu_{1,(i)}$. However, as mentioned in Remark~\ref{rem:transfer},
this is only a matter of interpretation: using the identity of
$z_{(i)}$ from~\eqref{eq:2id}, we can rewrite the transmission condition
 \[z_{1,(i)}^k(\alpha)=f_{\alpha,(i)}^{k-1},\quad \Longrightarrow \quad \dot\mu_{1,(i)}^k(\alpha)-d_i\mu_{1,(i)}^k(\alpha) = f_{\alpha,(i)}^{k-1},\] 
%$z_{1,(i)}^k(\alpha)=f_{\alpha,(i)}^{k-1}$ as $\dot\mu_{1,(i)}^k(\alpha)-d_i\mu_{1,(i)}^k(\alpha) = f_{\alpha,(i)}^{k-1}$, 
and define the update~\eqref{eq:DN2tran} as 
 \[f_{\alpha,(i)}^{k}:= (1-\theta)f_{\alpha,(i)}^{k-1} + \theta (\dot\mu_{2,(i)}^k(\alpha)-d_i\mu_{2,(i)}^k(\alpha)),\] 
%$f_{\alpha,(i)}^{k}:= (1-\theta)f_{\alpha,(i)}^{k-1} + \theta (\dot\mu_{2,(i)}^k(\alpha)-d_i\mu_{2,(i)}^k(\alpha))$, 
to rediscover the forward-backward structure. Moreover, with the interpretation of $\mu_{1,(i)}^k$, the algorithm DN$_2$~\eqref{eq:DN2} is a RN type algorithm.  

For the analysis, we choose the state $z_{(i)}$ formulation:
  for $i=1,\ldots,n$ and iteration index $k=1,2,\ldots,$ the
  equivalent algorithm reads
\begin{equation}\label{eq:errDN2}
  \left\{
    \begin{aligned}
      \ddot z_{1,(i)}^k - \sigma_i^2 z_{1,(i)}^k &= 0 \text{ in } \Omega_1,\\
      z_{1,(i)}^k(0) &= 0, \\
      z_{1,(i)}^k(\alpha) &= f_{\alpha,(i)}^{k-1},
    \end{aligned}
  \right.
  \quad
  \left\{
    \begin{aligned}
      \ddot z_{2,(i)}^k - \sigma_i^2 z_{2,(i)}^k &= 0 \text{ in } \Omega_2,\\
      \dot z_{2,(i)}^k(\alpha) &= \dot z_{1,(i)}^k(\alpha),\\
      \dot z_{2,(i)}^k(T) + \omega_i z_{2,(i)}^k(T) &= 0,
    \end{aligned}
  \right.
\end{equation}
where we still update the transmission condition by~\eqref{eq:DN2tran}. Note that~\eqref{eq:errDN2} is still a DN type algorithm, like~\eqref{eq:DN2}. Using the solutions~\eqref{eq:zsol} to determine the two coefficients $A_i^k$ and $B_i^k$, we get from~\eqref{eq:errDN2}
 \[A_i^k= \frac{f_{\alpha,(i)}^{k-1}}{\sinh(a_i)}, \quad B_i^k=-\frac{f_{\alpha,(i)}^{k-1}\coth(a_i)}{\sigma_i\sinh(b_i) + \omega_i\cosh(b_i) }.\]
%$A_i^k= \frac{f_{\alpha,(i)}^{k-1}}{\sinh(a_i)}$ and $B_i^k=-\frac{f_{\alpha,(i)}^{k-1}\coth(a_i)}{\sigma_i\sinh(b_i) + \omega_i\cosh(b_i) }$.
With~\eqref{eq:DN2tran}, we find
 \[f_{\alpha,(i)}^{k} = (1-\theta)f_{\alpha,(i)}^{k-1} - \theta f_{\alpha,(i)}^{k-1}\coth(a_i)\frac{\sigma_i\cosh(b_i) + \omega_i\sinh(b_i)}{\sigma_i\sinh(b_i) + \omega_i\cosh(b_i) },\]
%$f_{\alpha,(i)}^{k} = (1-\theta)f_{\alpha,(i)}^{k-1} - \theta f_{\alpha,(i)}^{k-1}\coth(a_i)\frac{\sigma_i\cosh(b_i) + \omega_i\sinh(b_i)}{\sigma_i\sinh(b_i) + \omega_i\cosh(b_i) }$,
and thus obtain the following convergence results.

\begin{theorem}
  The algorithm DN$_2$~\eqref{eq:DN2}-\eqref{eq:DN2tran} converges if and only if
  \begin{equation}\label{eq:rhoDN2}
    \rho_{\text{DN}_2}:=\max_{d_i\in\lambda(A)}\Big|1-\theta\Big(1+ \coth(a_i)\frac{\sigma_i\coth(b_i) + \omega_i}{\sigma_i + \omega_i\coth(b_i)}\Big)\Big|<1.
  \end{equation}
\end{theorem}
\begin{corollary}\label{cor:cvDN2cla}
  The algorithm DN$_2$ for $\theta=1$ does not converge if $\alpha\leq
  \frac T2$.
\end{corollary}
\begin{proof}
  Substituting $\theta=1$ into~\eqref{eq:rhoDN2}, we have
  \begin{equation}\label{eq:rhoDN_2}
    \rho_{\text{DN}_2}|_{\theta=1}=\max_{d_i\in\lambda(A)}\Big|\coth(a_i)\frac{\sigma_i\coth(b_i) + \omega_i}{\sigma_i + \omega_i\coth(b_i)}\Big|.
  \end{equation}
  Since $\coth(x)>1$, $\forall x>0$, both the numerator and the denominator in~\eqref{eq:rhoDN_2} are positive. When $a_i\leq b_i$ (i.e., $\alpha\leq T-\alpha$), we have $\coth(a_i)\geq\coth(b_i)$, and thus the difference between the numerator and the denominator is
   \[\begin{aligned}
     &\coth(a_i)\big(\omega_i + \sigma_i\coth(b_i)\big)-(\sigma_i + \omega_i\coth(b_i))\\
     =&\omega_i\big(\coth(a_i)-\coth(b_i)\big)+\sigma_i\big(\coth(b_i)\coth(a_i)-1\big)>0,
   \end{aligned}\]
%  $\coth(a_i)(\omega_i + \sigma_i\coth(b_i))-(\sigma_i + \omega_i\coth(b_i))=\omega_i(\coth(a_i)-\coth(b_i))+\sigma_i(\coth(b_i)\coth(a_i)-1)>0$,
  meaning that
   \[\coth(a_i)\frac{\sigma_i\coth(b_i) + \omega_i}{\sigma_i + \omega_i\coth(b_i)}>1.\]
%  $\coth(a_i)\frac{\sigma_i\coth(b_i) + \omega_i}{\sigma_i + \omega_i\coth(b_i)}>1$,
  which concludes the proof.
\end{proof}

We need some extra assumptions to conclude for the case $\alpha>\frac T2$.

\begin{corollary}\label{cor:cvDN2clabis}
  The algorithm DN$_2$ for $\theta=1$ does not converge if $\gamma=0$.
\end{corollary}

\begin{proof}
  We showed in Corollary~\ref{cor:cvDN2cla} the result for $\alpha\leq\frac T2$. Now $\alpha>\frac T2$ implies that $a_i>b_i$, thus $\coth(a_i)<\coth(b_i)$. Inserting $\gamma=0$ into~\eqref{eq:rhoDN_2} and using the definition of $\sigma_i$ from~\eqref{eq:sob}, the difference between the numerator and the denominator of~\eqref{eq:rhoDN_2} becomes
   \[\begin{aligned}
     &\coth(a_i)\big(d_i + \sigma_i\coth(b_i)\big)-(\sigma_i + d_i\coth(b_i))\\
     =&d_i\big(\coth(a_i)-\coth(b_i)\big)+\sigma_i\big(\coth(b_i)\coth(a_i)-1\big)\\
     =&\big(\coth(a_i)-\coth(b_i)\big)\big(d_i+\sigma_i\coth(b_i-a_i)\big)>0,
   \end{aligned}\]
%  $\coth(a_i)(d_i + \sigma_i\coth(b_i))-(\sigma_i + d_i\coth(b_i))=(\coth(a_i)-\coth(b_i))(d_i+\sigma_i\coth(b_i-a_i))>0,$
  where we use the fact that $d_i+\sigma_i\coth(b_i-a_i)<d_i-\sigma_i<0$. This shows that DN$_2$ for $\theta=1$ also does not converge for $\alpha>\frac T2$ when $\gamma=0$. 
\end{proof}

Unlike in Corollary~\ref{cor:cvDN1estim} where we have an estimate of the convergence factor for DN$_1$, we cannot provide a general convergence estimate for the algorithm DN$_2$~\eqref{eq:DN2}-\eqref{eq:DN2tran}, since we showed in Corollary~\ref{cor:cvDN2cla} and Corollary~\ref{cor:cvDN2clabis} that it does not converge in some cases. However, we can still show the convergence behavior for extreme eigenvalues. In particular, if the eigenvalue $d_i=0$, we find
\begin{equation}\label{eq:DN2d0}
  \begin{aligned}
    \rho_{\text{DN}_2}|_{d_i=0}=
    \textstyle\left|1-\theta\frac{\sqrt{\nu^{-1}}\cosh\left(\sqrt{\nu^{-1}}T\right) + \gamma\nu^{-1}\sinh\left(\sqrt{\nu^{-1}}T\right)}{\sinh(\sqrt{\nu^{-1}}\alpha)\left(\sqrt{\nu^{-1}}\sinh\left(\sqrt{\nu^{-1}} (T-\alpha)\right) + \gamma\nu^{-1}\cosh\left(\sqrt{\nu^{-1}} (T-\alpha)\right)\right) } \right|.
  \end{aligned}
\end{equation}
When the eigenvalue goes to infinity, using Remark~\ref{rem:DN1lim}, we obtain
 \[\lim_{\d_i\rightarrow\infty}\rho_{\text{DN}_2} = |1-2\theta|.\]
%$\lim_{\d_i\rightarrow\infty}\rho_{\text{DN}_2} = |1-2\theta|$. 
By equioscillating the convergence factor for small (i.e., $\rho_{\text{DN}_2}|_{d_i=0}$) and large eigenvalues (i.e., $\rho_{\text{DN}_2}|_{d_i\rightarrow\infty}$), we obtain after some computations
\begin{equation}\label{eq:thetaDN2opt}
  \begin{aligned}
    \theta^*_{\text{DN}_2}&=
    \frac2{2+\frac{\sqrt{\nu^{-1}}\cosh\left(\sqrt{\nu^{-1}}T\right) + \gamma\nu^{-1}\sinh\left(\sqrt{\nu^{-1}}T\right)}{\sinh(\sqrt{\nu^{-1}}\alpha)\left(\sqrt{\nu^{-1}}\sinh\left(\sqrt{\nu^{-1}} (T-\alpha)\right) + \gamma\nu^{-1}\cosh\left(\sqrt{\nu^{-1}} (T-\alpha)\right)\right) }}\\
    &\frac{2}{3+\coth(\sqrt{\nu^{-1}}\alpha)\frac{\coth\big(\sqrt{\nu^{-1}}(T-\alpha)\big) + \gamma\sqrt{\nu^{-1}}}{1 + \gamma\sqrt{\nu^{-1}}\coth\big(\sqrt{\nu^{-1}}(T-\alpha)\big)}}.
  \end{aligned}
\end{equation}

\begin{theorem}\label{thm:thetaDN2opt}
  If we assume that the eigenvalues of $A$ are anywhere in the interval $[0,\infty)$, then the optimal relaxation parameter $\theta^{\star}_{\text{DN}_2}$ for the algorithm DN$_2$~\eqref{eq:DN2}-\eqref{eq:DN2tran} with $\gamma=0$ is given by~\eqref{eq:thetaDN2opt} and satisfies $\theta^{\star}_{\text{DN}_2}<\frac12$.
\end{theorem}

\begin{proof}
  Taking the derivative of the convergence factor $\rho_{\text{DN}_2}$ from~\eqref{eq:rhoDN2} with respect to the eigenvalue $d_i$, we get 
%  $\frac{\d\rho_{\text{DN}_2}}{\d d_i} = -\frac{d_i\alpha}{\sigma_i\sinh^2(a_i)}\frac{\sigma_i\coth(b_i) + \omega_i}{\sigma_i + \omega_i\coth(b_i)} - \frac{\nu^{-1}\coth(a_i)}{\sigma_i}$ $\frac{\beta_i(\coth^2(b_i)-1)+\frac{d_i(T-\alpha)}{\sinh^2(b_i)}(1-\gamma^2\nu^{-1}-2d_i\gamma)}{(\sigma_i+ \omega_i\coth(b_i))^2}$, 
  \[\begin{aligned}
  \frac{\d\rho_{\text{DN}_2}}{\d d_i} &= -\frac{d_i\alpha}{\sigma_i\sinh^2(a_i)}\frac{\sigma_i\coth(b_i) + \omega_i}{\sigma_i + \omega_i\coth(b_i)} \\
  &- \frac{\nu^{-1}\coth(a_i)}{\sigma_i}\frac{\beta_i(\coth^2(b_i)-1)+\frac{d_i(T-\alpha)}{\sinh^2(b_i)}(1-\gamma^2\nu^{-1}-2d_i\gamma)}{(\sigma_i+ \omega_i\coth(b_i))^2}, 
  \end{aligned}\]
  where we used $\sigma_i,\omega_i$ and $\beta_i$ from~\eqref{eq:sob}. The derivative becomes negative with $\gamma=0$, meaning that the convergence factor decreases with respect to the eigenvalue $d_i$. We can then deduce the optimal relaxation parameter using equioscillation: inserting $\gamma=0$ into~\eqref{eq:thetaDN2opt}, the denominator becomes $3+\coth(\sqrt{\nu^{-1}}\alpha)\coth(\sqrt{\nu^{-1}}(T-\alpha))<4$.
\end{proof}

For $\gamma>0$, it is not clear when the convergence factor $\rho_{\text{DN}_2}$ is monotonic with respect to the eigenvalues, and thus the optimal relaxation parameter $\theta^{\star}_{\text{DN}_2}$ could differ from~\eqref{eq:thetaDN2opt}.

\subsubsection{Neumann-Dirichlet algorithm \texorpdfstring{(ND$_2$)}{ND2}}

We now invert the two conditions: for the iteration index $k=1,2,\ldots,$ the algorithm ND$_2$ to study is
\begin{equation}\label{eq:ND2}
  \begin{aligned}
    &\left\{
      \begin{aligned}
        \begin{pmatrix}
          \dot z_{1,(i)}^k\\
          \dot \mu_{1,(i)}^k
        \end{pmatrix}
        +
        \begin{pmatrix}
          d_i & -\nu^{-1} \\
          -1 & -d_i
        \end{pmatrix}
        \begin{pmatrix}
          z_{1,(i)}^k\\
          \mu_{1,(i)}^k
        \end{pmatrix}
        &=
        \begin{pmatrix}
          0\\
          0
        \end{pmatrix} \text{ in } \Omega_1,\\
        z_{1,(i)}^k(0) &= 0,\\
        \dot z_{1,(i)}^k(\alpha)& = f_{\alpha,(i)}^{k-1},
      \end{aligned}
    \right.\\
    &\left\{
      \begin{aligned}
        \begin{pmatrix}
          \dot z_{2,(i)}^k\\
          \dot \mu_{2,(i)}^k
        \end{pmatrix}
        +
        \begin{pmatrix}
          d_i & -\nu^{-1} \\
          -1 & -d_i
        \end{pmatrix}
        \begin{pmatrix}
          z_{2,(i)}^k\\
          \mu_{2,(i)}^k
        \end{pmatrix}
        &=
        \begin{pmatrix}
          0\\
          0
        \end{pmatrix} \text{ in } \Omega_2,\\
        z_{2,(i)}^k(\alpha) &= z_{1,(i)}^k(\alpha),\\
        \mu_{2,(i)}^k(T) + \gamma z_{2,(i)}^k(T)&=0,
      \end{aligned}
    \right.
  \end{aligned}
\end{equation}
and then we update the transmission condition by 
\begin{equation}\label{eq:ND2tran}
  f_{\alpha,(i)}^{k}:= (1-\theta)f_{\alpha,(i)}^{k-1} + \theta \dot z_{2,(i)}^k(\alpha),\quad \theta\in(0,1).
\end{equation}
Similar to the algorithm DN$_2$~\eqref{eq:DN2}-\eqref{eq:DN2tran}, we cannot see the forward-backward structure in $\Omega_1$ for the algorithm ND$_2$~\eqref{eq:ND2}-\eqref{eq:ND2tran}. But by interpreting the Neumann condition on $z_{1,(i)}$ in terms of $\mu_{1,(i)}$ as explained in Remark~\ref{rem:transfer}, the forward-backward structure is again revealed through a RD type algorithm.

We proceed for the convergence analysis using the formulation~\eqref{eq:z}: for $i=1,\ldots,n$ and iteration index $k=1,2,\ldots,$ we solve
\begin{equation}\label{eq:errND2}
  \left\{
    \begin{aligned}
      \ddot z_{1,(i)}^k - (d_i^2+\nu^{-1}) z_{1,(i)}^k &= 0 \text{ in } \Omega_1,\\
      z_{1,(i)}^k(0) &= 0, \\
      \dot z_{1,(i)}^k(\alpha) &= f_{\alpha,(i)}^{k-1},
    \end{aligned}
  \right.
  \quad
  \left\{
    \begin{aligned}
      \ddot z_{2,(i)}^k - (d_i^2+\nu^{-1}) z_{2,(i)}^k &= 0 \text{ in } \Omega_2,\\
      z_{2,(i)}^k(\alpha) &= z_{1,(i)}^k(\alpha),\\
      \dot z_{2,(i)}^k(T) + d_i z_{2,(i)}^k(T) &= -\gamma\nu^{-1}z_{2,(i)}^k(T),
    \end{aligned}
  \right.
\end{equation}
where we still update the transmission condition by~\eqref{eq:ND2tran}. Note that both algorithms~\eqref{eq:ND2} and~\eqref{eq:errND2} are of ND type.

Using the solutions~\eqref{eq:zsol} and the transmission condition in~\eqref{eq:ND2tran}, we obtain
 \[A_i^k= \frac{f_{\alpha,(i)}^{k-1}}{\sigma_i\cosh(a_i)}, \quad B_i^k=\frac{f_{\alpha,(i)}^{k-1}\tanh(a_i)/\sigma_i}{\sigma_i\cosh(b_i) + \omega_i\sinh(b_i) }.\]
%$A_i^k= \frac{f_{\alpha,(i)}^{k-1}}{\sigma_i\cosh(a_i)}$, $B_i^k=\frac{f_{\alpha,(i)}^{k-1}\tanh(a_i)/\sigma_i}{\sigma_i\cosh(b_i) + \omega_i\sinh(b_i) }$,
and we therefore get for the update condition~\eqref{eq:ND2tran}
 \[f_{\alpha,(i)}^{k} = (1-\theta)f_{\alpha,(i)}^{k-1} - \theta f_{\alpha,(i)}^{k-1}\tanh(a_i)\frac{\sigma_i\sinh(b_i) + \omega_i\cosh(b_i)}{\sigma_i\cosh(b_i) + \omega_i\sinh(b_i) }.\]
%$f_{\alpha,(i)}^{k} = (1-\theta)f_{\alpha,(i)}^{k-1} - \theta f_{\alpha,(i)}^{k-1}\tanh(a_i)\frac{\sigma_i\sinh(b_i) + \omega_i\cosh(b_i)}{\sigma_i\cosh(b_i) + \omega_i\sinh(b_i) }$.
\begin{theorem}
  The algorithm ND$_2$~\eqref{eq:ND2}-\eqref{eq:ND2tran} converges if and only if
  \begin{equation}\label{eq:rhoND2}
%     \rho_{\text{ND}_2}:=\max_{d_i\in\lambda(A)}\Big|1-\theta\frac{\sigma_i\cosh\left(\sigma_i T\right) + \omega_i\sinh\left(\sigma_i T\right)}{\cosh(a)\big(\sigma_i\cosh(b) + \omega_i\sinh(b)\big) }\Big|<1,
    \rho_{\text{ND}_2}:=\max_{d_i\in\lambda(A)}\Big|1-\theta\Big(1+\tanh(a_i)\frac{\sigma_i\tanh(b_i) + \omega_i}{\sigma_i + \omega_i\tanh(b_i) }\Big)\Big|<1.
  \end{equation}
\end{theorem}
\begin{corollary}\label{cor:cvND2cla}
  The algorithm ND$_2$ for $\theta=1$ converges if $\alpha\leq \frac T2$.
\end{corollary}
\begin{proof}
  Substituting $\theta=1$ into~\eqref{eq:rhoND2}, we have
  \begin{equation}\label{eq:rhoND_2}
    \rho_{\text{ND}_2}|_{\theta=1}=\max_{d_i\in\lambda(A)}\Big|\tanh(a_i)\frac{\sigma_i\tanh(b_i) + \omega_i}{\sigma_i + \omega_i\tanh(b_i)}\Big|.
  \end{equation}
  Since $0<\tanh(x)<1$, $\forall x>0$, both the numerator and the denominator in~\eqref{eq:rhoND_2} are positive. In the case where $a_i\leq b_i$ (i.e., $\alpha\leq T-\alpha$), we have $\tanh(a_i)\leq\tanh(b_i)$, and the difference between the numerator and the denominator is
   \[\begin{aligned}
     &\tanh(a_i)\big(\omega_i + \sigma_i\tanh(b_i)\big)-(\sigma_i + \omega_i\tanh(b_i))\\
     =&\omega_i\big(\tanh(a_i)-\tanh(b_i)\big)+\sigma_i\big(\tanh(b_i)\tanh(a_i)-1\big)<0,
   \end{aligned}\]
%  $\tanh(a_i)(\omega_i + \sigma_i\tanh(b_i))-(\sigma_i + \omega_i\tanh(b_i))=\omega_i(\tanh(a_i)-\tanh(b_i))+\sigma_i(\tanh(b_i)\tanh(a_i)-1)<0$,
  meaning that
   \[0<\tanh(a_i)\frac{\sigma_i\tanh(b_i) + \omega_i}{\sigma_i + \omega_i\tanh(b_i)}<1.\] 
%  $0<\tanh(a_i)\frac{\sigma_i\tanh(b_i) + \omega_i}{\sigma_i + \omega_i\tanh(b_i)}<1$.
  This concludes the proof.
\end{proof}

As shown in Corollary~\ref{cor:cvDN2cla}, the algorithm DN$_2$~\eqref{eq:DN2}-\eqref{eq:DN2tran} with $\theta=1$ does not converge for $\alpha\leq \frac T2$, whereas the algorithm ND$_2$~\eqref{eq:ND2}-\eqref{eq:ND2tran} converges in this case. This reveals a symmetry behavior, since the only difference between these two algorithms is that we exchange the Dirichlet and the Neumann condition in the two subdomains. This symmetry is well-known for classical DN and ND algorithms. 

\begin{corollary}
  For $\gamma=0$, the algorithm ND$_2$ for $\theta=1$ converges for all initial guesses.
\end{corollary}

\begin{proof}
  This is shown in Corollary~\ref{cor:cvND2cla} for $\alpha\leq\frac T2$. If $\alpha>\frac T2$, i.e. $a_i>b_i$, then $\tanh(a_i)>\tanh(b_i)$, and the difference between the numerator and the denominator is
   \[\begin{aligned}
     &\tanh(a_i)\big(d_i + \sigma_i\tanh(b_i)\big)-(\sigma_i + d_i\tanh(b_i))\\
     =&d_i\big(\tanh(a_i)-\tanh(b_i)\big)+\sigma_i\big(\tanh(b_i)\tanh(a_i)-1\big)\\
     =&\big(\tanh(b_i)\tanh(a_i)-1\big)\big(\sigma_i-d_i\tanh(a_i-b_i)\big)<0,
   \end{aligned}\]
%  $\tanh(a_i)(d_i + \sigma_i\tanh(b_i))-(\sigma_i + d_i\tanh(b_i))=(\tanh(b_i)\tanh(a_i)-1)(\sigma_i-d_i\tanh(a_i-b_i))<0$,
  where we use the fact that $0<\sigma_i-d_i<\sigma_i-d_i\tanh(a_i-b_i)$. This shows that the algorithm ND$_2$ for $\theta=1$ converge for $\alpha>\frac T2$ in the case $\gamma=0$.
\end{proof} 

Notice that the matrix $A$ here can be singular, in contrast to the algorithm DN$_1$ in Corollary~\ref{cor:cvDN1cla} where non-singularity is needed for $A$. As in the previous section, we can still show the convergence behavior for extreme eigenvalues. If the eigenvalue $d_i=0$, we find
\begin{equation}\label{eq:ND2d0}
  \begin{aligned}
    \rho_{\text{ND}_2}|_{d_i=0}=
    \textstyle\left|1-\theta\frac{\sqrt{\nu^{-1}}\cosh\left(\sqrt{\nu^{-1}}T\right) + \gamma\nu^{-1}\sinh\left(\sqrt{\nu^{-1}}T\right)}{\cosh(\sqrt{\nu^{-1}}\alpha)\left(\sqrt{\nu^{-1}}\cosh\left(\sqrt{\nu^{-1}} (T-\alpha)\right) + \gamma\nu^{-1}\sinh\left(\sqrt{\nu^{-1}} (T-\alpha)\right)\right) } \right|.
  \end{aligned}
\end{equation}
The expression~\eqref{eq:ND2d0} is very similar to~\eqref{eq:DN2d0}: when $\gamma=0$, the convergence factor~\eqref{eq:DN2d0} becomes
 \[\rho_{\text{DN}_2}|_{d_i=0,\gamma=0}=\left|1-\theta\Big(1+\coth\left(\sqrt{\nu^{-1}}\alpha\right)\coth\left(\sqrt{\nu^{-1}}(T-\alpha)\right)\Big) \right|,\]
%$\rho_{\text{DN}_2}|_{d_i=0,\gamma=0}=|1-\theta(1+\coth(\sqrt{\nu^{-1}}\alpha)\coth(\sqrt{\nu^{-1}}$ $(T-\alpha)))|$,
whereas~\eqref{eq:ND2d0} becomes
 \[\rho_{\text{ND}_2}|_{d_i=0,\gamma=0}=\left|1-\theta\Big(1+\tanh\left(\sqrt{\nu^{-1}}\alpha\right)\tanh\left(\sqrt{\nu^{-1}}(T-\alpha)\right)\Big) \right|.\]
%$\rho_{\text{ND}_2}|_{d_i=0,\gamma=0}=|1-\theta(1+\tanh(\sqrt{\nu^{-1}}\alpha)$ $\tanh(\sqrt{\nu^{-1}}(T-\alpha)))|$.
We find again the symmetry between DN$_2$ and ND$_2$. In the case when the eigenvalue goes to infinity, using Remark~\ref{rem:DN1lim}, we obtain
 \[\lim_{\d_i\rightarrow\infty}\rho_{\text{ND}_2} = |1-2\theta|,\]
%$\lim_{\d_i\rightarrow\infty}\rho_{\text{ND}_2} = |1-2\theta|$,
as for DN$_2$. By equioscillating the convergence factor again for small and large eigenvalues, we obtain after some computations the relaxation parameter
\begin{equation}\label{eq:thetaND2opt}
  \begin{aligned}
    \theta^*_{\text{ND}_2}
    &=\frac2{2+\frac{\sqrt{\nu^{-1}}\cosh\left(\sqrt{\nu^{-1}}T\right) + \gamma\nu^{-1}\sinh\left(\sqrt{\nu^{-1}}T\right)}{\cosh(\sqrt{\nu^{-1}}\alpha)\left(\sqrt{\nu^{-1}}\cosh\left(\sqrt{\nu^{-1}} (T-\alpha)\right) + \gamma\nu^{-1}\sinh\left(\sqrt{\nu^{-1}} (T-\alpha)\right)\right) }}\\
    &=\frac{2}{3+\tanh(\sqrt{\nu^{-1}}\alpha)\frac{\tanh\left(\sqrt{\nu^{-1}}(T-\alpha)\right) + \gamma\sqrt{\nu^{-1}}}{1 + \gamma\sqrt{\nu^{-1}}\tanh\left(\sqrt{\nu^{-1}}(T-\alpha)\right)}}.
  \end{aligned}
\end{equation}
We thus obtain a similar result as Theorem~\ref{thm:thetaDN2opt}.

\begin{theorem}\label{thm:thetaND2opt}
  If we assume that the eigenvalues of $A$ are anywhere in the interval $[0,\infty)$, then the optimal relaxation parameter $\theta^{\star}_{\text{ND}_2}$ for the algorithm ND$_2$~\eqref{eq:ND2}-\eqref{eq:ND2tran} with $\gamma=0$ is given by~\eqref{eq:thetaND2opt}, and satisfies $\frac12<\theta^{\star}_{\text{ND}_2}<\frac23$.
\end{theorem}

\begin{proof}
  As for Theorem~\ref{thm:thetaDN2opt}, we take the derivative of $\rho_{\text{ND}_2}$ with respect to $d_i$,
%  $\frac{\d\rho_{\text{ND}_2}}{\d d_i} =
% 	\frac{d_i\alpha}{\sigma_i\cosh^2(a_i)}\frac{\sigma_i\tanh(b_i) +
%    \omega_i}{\sigma_i + \omega_i\tanh(b_i)} +
%  \frac{\nu^{-1}\tanh(a_i)}{\sigma_i}\frac{\beta_i(1-\tanh^2(b_i))-\frac{d_i(T-\alpha)}{\cosh^2(b_i)}(\gamma^2\nu^{-1}+2d_i\gamma
%    -1)}{(\sigma_i + \omega_i\tanh(b_i))^2}$, 
    \[\begin{aligned}
    \frac{\d\rho_{\text{ND}_2}}{\d d_i} =&
 	\frac{d_i\alpha}{\sigma_i\cosh^2(a_i)}\frac{\sigma_i\tanh(b_i) +
    \omega_i}{\sigma_i + \omega_i\tanh(b_i)} \\
    &+\frac{\nu^{-1}\tanh(a_i)}{\sigma_i}\frac{\beta_i(1-\tanh^2(b_i))-\frac{d_i(T-\alpha)}{\cosh^2(b_i)}(\gamma^2\nu^{-1}+2d_i\gamma
    -1)}{(\sigma_i + \omega_i\tanh(b_i))^2},
    \end{aligned}\]
    with $\sigma_i,\omega_i$
  and $\beta_i$ defined in~\eqref{eq:sob}. For $\gamma=0$, the
    derivative is positive and thus $\rho_{\text{ND}_2}$ increases
    with $d_i$. Therefore $\theta^*_{\text{ND}_2}$
   is determined by equioscillation. Inserting $\gamma=0$
  into~\eqref{eq:thetaND2opt}, the denominator becomes
  $3+\tanh(\sqrt{\nu^{-1}}\alpha)\tanh(\sqrt{\nu^{-1}}(T-\alpha))<4$.
\end{proof}

As for DN$_2$ however, the monotonicity of the
convergence factor $\rho_{\text{ND}_2}$ is not guaranteed
  for $\gamma>0$, and the optimal relaxation parameter
$\theta^{\star}_{\text{ND}_2}$ may differ from~\eqref{eq:thetaND2opt}.

\subsection{Category III}

We finally study algorithms in Category III which run only on the
state $\mu_{(i)}$ to solve the problem~\eqref{eq:sysODEreduced},
and use DN and ND techniques.

\subsubsection{Dirichlet-Neumann algorithm \texorpdfstring{(DN$_3$)}{DN3}}
As shown in Table~\ref{tab:combination}, we apply the Dirichlet
condition in $\Omega_1$ and the Neumann condition in $\Omega_2$, both
to the state $\mu_{(i)}$. For iteration index $k=1,2,\ldots,$ the algorithm DN$_3$ solves
\begin{equation}\label{eq:DN3}
  \begin{aligned}
    &\left\{
      \begin{aligned}
        \begin{pmatrix}
          \dot z_{1,(i)}^k\\
          \dot \mu_{1,(i)}^k
        \end{pmatrix}
        +
        \begin{pmatrix}
          d_i & -\nu^{-1} \\
          -1 & -d_i
        \end{pmatrix}
        \begin{pmatrix}
          z_{1,(i)}^k\\
          \mu_{1,(i)}^k
        \end{pmatrix}
        &=
        \begin{pmatrix}
          0\\
          0
        \end{pmatrix} \text{ in } \Omega_1,\\
        z_{1,(i)}^k(0) &= 0,\\
        \mu_{1,(i)}^k(\alpha) & = f_{\alpha,(i)}^{k-1},
      \end{aligned}
    \right.\\
    &\left\{
      \begin{aligned}
        \begin{pmatrix}
          \dot z_{2,(i)}^k\\
          \dot \mu_{2,(i)}^k
        \end{pmatrix}
        +
        \begin{pmatrix}
          d_i & -\nu^{-1} \\
          -1 & -d_i
        \end{pmatrix}
        \begin{pmatrix}
          z_{2,(i)}^k\\
          \mu_{2,(i)}^k
        \end{pmatrix}
        &=
        \begin{pmatrix}
          0\\
          0
        \end{pmatrix} \text{ in } \Omega_2,\\
        \dot \mu_{2,(i)}^k(\alpha)&=\dot \mu_{1,(i)}^k(\alpha),\\
        \mu_{2,(i)}^k(T) + \gamma z_{2,(i)}^k(T)&=0,
      \end{aligned}
    \right.
  \end{aligned}
\end{equation}
and we update the transmission condition by 
\begin{equation}\label{eq:DN3tran}
  f_{\alpha,(i)}^{k}:= (1-\theta)f_{\alpha,(i)}^{k-1} + \theta \mu_{2,(i)}^k(\alpha), \theta\in(0,1).
\end{equation}
The forward-backward structure is now less present in $\Omega_2$,
  where we would expect to have an initial condition for $z_{2,(i)}$ instead
  of $\mu_{2,(i)}$. By using the identity of $\mu_{(i)}$
  in~\eqref{eq:2id}, we can interpret the Neumann condition
 \[\dot \mu_{2,(i)}^k(\alpha)=\dot \mu_{1,(i)}^k(\alpha),\quad \Longrightarrow \quad d_i\dot z_{2,(i)}^k(\alpha) + \sigma_i^2 z_{2,(i)}^k(\alpha) = d_i\dot z_{1,(i)}^k(\alpha) + \sigma_i^2 z_{1,(i)}^k(\alpha),\]
%$\dot \mu_{2,(i)}^k(\alpha)=\dot \mu_{1,(i)}^k(\alpha)$ as $d_i\dot z_{2,(i)}^k(\alpha) + \sigma_i^2 z_{2,(i)}^k(\alpha) = d_i\dot z_{1,(i)}^k(\alpha) + \sigma_i^2 z_{1,(i)}^k(\alpha)$,
a Robin type condition on $z_{2,(i)}$. Therefore, the algorithm DN$_3$ can also
be interpreted as a DR algorithm.

For the convergence analysis, it is natural to choose the
interpretation in $\mu_{(i)}$, i.e., using~\eqref{eq:mu}, which gives
\begin{equation}\label{eq:errDN3}
  \left\{
    \begin{aligned}
      \ddot \mu_{1,(i)}^k - \sigma_i^2 \mu_{1,(i)}^k &= 0 \text{ in } \Omega_1,\\
      \dot \mu_{(i)}(0)-d_i\mu_{(i)}(0) &= 0, \\
      \mu_{1,(i)}^k(\alpha) &= f_{\alpha,(i)}^{k-1},
    \end{aligned}
  \right.
  \quad
  \left\{
    \begin{aligned}
      \ddot \mu_{2,(i)}^k - \sigma_i^2 \mu_{2,(i)}^k &= 0 \text{ in } \Omega_2,\\
      \dot \mu_{2,(i)}^k(\alpha) &= \dot \mu_{1,(i)}^k(\alpha),\\
      \gamma \dot \mu_{(i)}(T)+\beta_i\mu_{(i)}(T)&=0,
    \end{aligned}
  \right.
\end{equation}
where we still update the transmission condition
through~\eqref{eq:DN3tran}. We observe that both~\eqref{eq:DN3}
and~\eqref{eq:errDN3} are DN type algorithms. Proceeding as before, 
we obtain:
% and we obtain 
% \[\begin{aligned}
%   A_i^k&= \frac{f_{\alpha,(i)}^{k-1}}{\sigma_i\cosh(a_i) + d_i\sinh(a_i)}\\
%   B_i^k&=\frac{-f_{\alpha,(i)}^{k-1}}{\gamma\sigma_i\sinh(b_i) + \beta_i\cosh(b_i) }\frac{\sigma_i\sinh(a_i) + d_i\cosh(a_i)}{\sigma_i\cosh(a_i) + d_i\sinh(a_i)}.
% \end{aligned}\]
% Using then the relation~\eqref{eq:DN3tran}, we find
% \[\begin{aligned}
%   f_{\alpha,(i)}^{k} = (1-\theta)f_{\alpha,(i)}^{k-1} - \theta f_{\alpha,(i)}^{k-1}\frac{\sigma_i\sinh(a_i) + d_i\cosh(a_i)}{\sigma_i\cosh(a_i) + d_i\sinh(a_i)}\frac{\gamma\sigma_i\cosh(b_i) + \beta_i\sinh(b_i)}{\gamma\sigma_i\sinh(b_i) + \beta_i\cosh(b_i) }.
% \end{aligned}\]
% W thus obtain the following results.

\begin{theorem}
  The algorithm DN$_3$~\eqref{eq:DN3}-\eqref{eq:DN3tran} converges if and only if
  \begin{equation}\label{eq:rhoDN3}
    % \rho_{\text{DN}_3}:=\max_{d_i\in\lambda(A)}\Big|1-\theta\frac{\sigma_i\cosh(\sigma_i T)+\omega_i\sinh(\sigma_i T)}{\big(\sigma_i\cosh(a) + d_i\sinh(a)\big)\big(\gamma\sigma_i\sinh(b) + \beta_i\cosh(b)\big)}\Big|<1,
    \rho_{\text{DN}_3}:=\max_{d_i\in\lambda(A)}\Big|1-\theta\Big(1+\frac{\sigma_i + d_i\coth(a_i)}{\sigma_i\coth(a_i) + d_i}\frac{\gamma\sigma_i\coth(b_i) + \beta_i}{\gamma\sigma_i + \beta_i\coth(b_i)}\Big)\Big|<1.
  \end{equation}
\end{theorem}

To get more insight, we choose $\theta=1$ in~\eqref{eq:rhoDN3}, and find
\begin{equation}\label{eq:rhoDN_3}
  \rho_{\text{DN}_3}|_{\theta=1}=\max_{d_i\in\lambda(A)}\Big|\frac{\sigma_i + d_i\coth(a_i)}{\sigma_i\coth(a_i) + d_i}\frac{\gamma\sigma_i\coth(b_i) + \beta_i}{\gamma\sigma_i + \beta_i\coth(b_i)}\Big|.
\end{equation}
It is less clear whether $\gamma\sigma_i + \beta_i\coth(b_i)$ is
positive, since, using the definition of $\beta_i$ and $\sigma_i$
from~\eqref{eq:sob}, we have
 \[\begin{aligned}
   \gamma\sigma_i + \beta_i\coth(b_i) = \gamma\Big(\sqrt{d_i^2+\nu^{-1}}-d_i\coth\big(\sqrt{d_i^2+\nu^{-1}}&(T-\alpha)\big)\Big) \\
   + &\coth\big(\sqrt{d_i^2+\nu^{-1}}(T-\alpha)\big),
 \end{aligned}\]
%$\gamma\sigma_i + \beta_i\coth(b_i) = \gamma(\sqrt{d_i^2+\nu^{-1}}-d_i\coth(\sqrt{d_i^2+\nu^{-1}}(T-\alpha))) + \coth(\sqrt{d_i^2+\nu^{-1}}(T-\alpha))$,
and depending on the values of $\nu,\gamma$ and $\alpha$, this could
be negative. However, we can simplify~\eqref{eq:rhoDN_3} by setting
$\gamma=0$, and obtain:
\begin{corollary}\label{cor:cvDN3cla}
  If $\gamma=0$, then the algorithm DN$_3$ with $\theta=1$ converges for all initial guesses.
\end{corollary}

\begin{proof}
  Substituting $\theta=1$ into~\eqref{eq:rhoDN_3}, we have
  \begin{equation}\label{eq:rhoDN_3gam0}
    \rho_{\text{DN}_3}|_{\theta=1}=\max_{d_i\in\lambda(A)}\Big|\frac{\sigma_i\tanh(a_i) + d_i}{\sigma_i + d_i\tanh(a_i)}\tanh(b_i)\Big|.
  \end{equation}
  Both the numerator and the denominator are positive. Using $0<\tanh(x)<1$, $\forall x>0$, we get
   \[\big(d_i + \sigma_i\tanh(a_i)\big)-(\sigma_i + d_i\tanh(a_i))=\big(d_i-\sigma_i\big)\big(1-\tanh(a_i)\big)<0.\]
%  $(d_i + \sigma_i\tanh(a_i))-(\sigma_i + d_i\tanh(a_i))=(d_i-\sigma_i)(1-\tanh(a_i))<0$, meaning that
   \[0<\tanh(b_i)\frac{\sigma_i\tanh(a_i) + d_i}{\sigma_i + d_i\tanh(a_i)}<1.\] 
%  $0<\tanh(b_i)\frac{\sigma_i\tanh(a_i) + d_i}{\sigma_i + d_i\tanh(a_i)}<1$,
  which concludes the proof.
\end{proof}

For $\gamma=0$, the algorithm
  DN$_3$~\eqref{eq:DN3}-\eqref{eq:DN3tran} converges for $\theta=1$ as
  well as the algorithm ND$_2$~\eqref{eq:ND2}-\eqref{eq:ND2tran},
  since their convergence factors are very similar. For
    extreme eigenvalues, inserting $d_i=0$ into~\eqref{eq:rhoDN3}, we
    find the identical formula as~\eqref{eq:ND2d0}, and when the
  eigenvalue goes to infinity, we also obtain
 \[\lim_{\d_i\rightarrow\infty}\rho_{\text{DN}_3} = |1-2\theta|,\]
%$\lim_{\d_i\rightarrow\infty}\rho_{\text{DN}_3} = |1-2\theta|$. 
By equioscillating the convergence
factor between small and large eigenvalues, we obtain thus the same
relaxation parameter as~\eqref{eq:thetaND2opt}, which leads to:

\begin{theorem}\label{thm:thetaDN3opt}
  If we assume the eigenvalues of $A$ can be anywhere in the interval
  $[0,\infty)$, then the optimal relaxation parameter
    $\theta^{\star}_{\text{DN}_3}$ for the algorithm
    DN$_3$~\eqref{eq:DN3}-\eqref{eq:DN3tran} with $\gamma=0$ is
    identical to $\theta^{\star}_{\text{ND}_2}$.
\end{theorem}

\begin{proof}
  For $\gamma=0$, the convergence factors~\eqref{eq:rhoND_2}
  and~\eqref{eq:rhoDN_3gam0} become the same when exchanging
  $a_i$ and $b_i$, and the result thus follows as for
    Theorem~\ref{thm:thetaND2opt}.
\end{proof}

\subsubsection{Neumann-Dirichlet algorithm \texorpdfstring{(ND$_3$)}{ND3}}

We now exchange the Dirichlet and Neumann conditions on the two
subdomains, and obtain
\begin{equation}\label{eq:ND3}
  \begin{aligned}
    &\left\{
      \begin{aligned}
        \begin{pmatrix}
          \dot z_{1,(i)}^k\\
          \dot \mu_{1,(i)}^k
        \end{pmatrix}
        +
        \begin{pmatrix}
          d_i & -\nu^{-1} \\
          -1 & -d_i
        \end{pmatrix}
        \begin{pmatrix}
          z_{1,(i)}^k\\
          \mu_{1,(i)}^k
        \end{pmatrix}
        &=
        \begin{pmatrix}
          0\\
          0
        \end{pmatrix} \text{ in } \Omega_1,\\
        z_{1,(i)}^k(0) &= 0,\\
        \dot\mu_{1,(i)}^k(\alpha) & = f_{\alpha,(i)}^{k-1},
      \end{aligned}
    \right.\\
    &\left\{
      \begin{aligned}
        \begin{pmatrix}
          \dot z_{2,(i)}^k\\
          \dot \mu_{2,(i)}^k
        \end{pmatrix}
        +
        \begin{pmatrix}
          d_i & -\nu^{-1} \\
          -1 & -d_i
        \end{pmatrix}
        \begin{pmatrix}
          z_{2,(i)}^k\\
          \mu_{2,(i)}^k
        \end{pmatrix}
        &=
        \begin{pmatrix}
          0\\
          0
        \end{pmatrix} \text{ in } \Omega_2,\\
        \mu_{2,(i)}^k(\alpha)&=\mu_{1,(i)}^k(\alpha),\\
        \mu_{2,(i)}^k(T) + \gamma z_{2,(i)}^k(T)&=0,
      \end{aligned}
    \right.
  \end{aligned}
\end{equation}
where the transmission condition is updated by
\begin{equation}\label{eq:ND3tran}
  f_{\alpha,(i)}^{k}:= (1-\theta)f_{\alpha,(i)}^{k-1} + \theta \dot \mu_{2,(i)}^k(\alpha),\theta\in(0,1).
\end{equation}
As for DN$_3$, we need to use the identity~\eqref{eq:2id} and
interpret $\mu_{2,(i)}^k(\alpha)=\mu_{1,(i)}^k(\alpha)$ as
 \[\dot z_{2,(i)}^k(\alpha) + d_i z_{2,(i)}^k(\alpha) = \dot z_{1,(i)}^k(\alpha) + d_i z_{1,(i)}^k(\alpha),\] 
%$\dot z_{2,(i)}^k(\alpha) + d_i z_{2,(i)}^k(\alpha) = \dot z_{1,(i)}^k(\alpha) + d_i z_{1,(i)}^k(\alpha)$
to reveal the forward-backward structure with a NR type
algorithm. Using formulation~\eqref{eq:mu}, we get
\begin{equation}\label{eq:errND3}
  \left\{
    \begin{aligned}
      \ddot \mu_{1,(i)}^k - \sigma_i^2 \mu_{1,(i)}^k &= 0 \text{ in } \Omega_1,\\
      \dot \mu_{(i)}(0)-d_i\mu_{(i)}(0) &= 0, \\
      \dot \mu_{1,(i)}^k(\alpha) &= f_{\alpha,(i)}^{k-1},
    \end{aligned}
  \right.
  \quad
  \left\{
    \begin{aligned}
      \ddot \mu_{2,(i)}^k - \sigma_i^2 \mu_{2,(i)}^k &= 0 \text{ in } \Omega_2,\\
      \mu_{2,(i)}^k(\alpha) &= \mu_{1,(i)}^k(\alpha),\\
      \gamma \dot \mu_{(i)}(T)+\beta_i\mu_{(i)}(T)&=0.
    \end{aligned}
  \right.
\end{equation}
% and we obtain 
% \[\begin{aligned}
%   A_i^k&= \frac{f_{\alpha,(i)}^{k-1}/\sigma_i}{\sigma_i\sinh(a_i) + d_i\cosh(a_i)}\\
%   B_i^k&=\frac{f_{\alpha,(i)}^{k-1}/\sigma_i}{\gamma\sigma_i\cosh(b_i) + \beta_i\sinh(b_i) }\frac{\sigma_i\cosh(a_i) + d_i\sinh(a_i)}{\sigma_i\sinh(a_i) + d_i\cosh(a_i)}.
% \end{aligned}\]
% Using then the relation~\eqref{eq:ND3tran}, we find
% \[\begin{aligned}
%   f_{\alpha,(i)}^{k} = (1-\theta)f_{\alpha,(i)}^{k-1} - \theta f_{\alpha,(i)}^{k-1}\frac{\sigma_i\cosh(a_i) + d_i\sinh(a_i)}{\sigma_i\sinh(a_i) + d_i\cosh(a_i)}\frac{\gamma\sigma_i\sinh(b_i) + \beta_i\cosh(b_i)}{\gamma\sigma_i\cosh(b_i) + \beta_i\sinh(b_i)}.
% \end{aligned}\]
% In this way, we have the following results.

\begin{theorem}
  The algorithm ND$_3$~\eqref{eq:ND3}-\eqref{eq:ND3tran} converges if
  and only if
  \begin{equation}\label{eq:rhoND3}
    % \rho_{\text{ND}_3}:=\max_{d_i\in\lambda(A)}\Big|1-\theta\frac{\sigma_i\cosh(\sigma_i T)+\omega_i\sinh(\sigma_i T)}{\big(\sigma_i\sinh(a) + d_i\cosh(a)\big)\big(\gamma\sigma_i\cosh(b) - (\gamma d_i-1)\sinh(b)\big)}\Big|<1,
    \rho_{\text{ND}_3}:=\max_{d_i\in\lambda(A)}\Big|1-\theta\Big(1+\frac{\sigma_i + d_i\tanh(a_i)}{\sigma_i\tanh(a_i) + d_i}\frac{\gamma\sigma_i\tanh(b_i) + \beta_i}{\gamma\sigma_i + \beta_i\tanh(b_i)}\Big)\Big|<1.
  \end{equation}
\end{theorem}

As in the previous section, we choose $\theta=1$ in~\eqref{eq:rhoND3}, and find
\begin{equation}\label{eq:rhoND_3}
  \rho_{\text{ND}_3}|_{\theta=1}=\max_{d_i\in\lambda(A)}\Big|\frac{\sigma_i + d_i\tanh(a_i)}{\sigma_i\tanh(a_i) + d_i}\frac{\gamma\sigma_i\tanh(b_i) + \beta_i}{\gamma\sigma_i + \beta_i\tanh(b_i)}\Big|.
\end{equation}
Again, using the definition of $\beta_i$ and $\sigma_i$
from~\eqref{eq:sob}, we have
 \[\gamma\sigma_i\tanh(b_i) + \beta_i = \gamma\Big(\sqrt{d_i^2+\nu^{-1}}\tanh\big(\sqrt{d_i^2+\nu^{-1}}(T-\alpha)\big)-d_i\Big) + 1,\]
%$\gamma\sigma_i\tanh(b_i) + \beta_i = \gamma(\sqrt{d_i^2+\nu^{-1}}\tanh(\sqrt{d_i^2+\nu^{-1}}(T-\alpha))-d_i) + 1$,
and depending on the values of $\nu,\gamma$ and $\alpha$, this could be negative. However, we can simplify~\eqref{eq:rhoND_3} by taking $\gamma=0$, and then obtain the following result.

\begin{corollary}\label{cor:cvND3cla}
  If $\gamma=0$, then the algorithm ND$_3$ with $\theta=1$ does not converge.
\end{corollary}

\begin{proof}
  Inserting $\gamma=0$ into~\eqref{eq:rhoND_3}, we get
  \begin{equation}\label{eq:rhoND_3gam0}
    \rho_{\text{DN}_3}|_{\theta=1}=\max_{d_i\in\lambda(A)}\Big|\frac{\sigma_i\coth(a_i) + d_i}{\sigma_i + d_i\coth(a_i)}\coth(b_i)\Big|.
  \end{equation}
  Both the numerator and the denominator are positive. Using $\coth(x)\geq1$, $\forall x>0$, we find
   \[\big(d_i + \sigma_i\coth(a_i)\big)-(\sigma_i + d_i\coth(a_i))=\big(\sigma_i-d_i\big)\big(\coth(a_i)-1\big)>0.\]
%  $(d_i + \sigma_i\coth(a_i))-(\sigma_i + d_i\coth(a_i))=(\sigma_i-d_i)(\coth(a_i)-1)>0$, 
  implying that
   \[\frac{\sigma_i\coth(a_i) + d_i}{\sigma_i + d_i\coth(a_i)}\coth(b_i)>1.\] 
%  $\frac{\sigma_i\coth(a_i) + d_i}{\sigma_i + d_i\coth(a_i)}\coth(b_i)>1$, 
  which concludes the proof.
\end{proof}

Comparing Corollaries~\ref{cor:cvDN3cla} and \ref{cor:cvND3cla},
  we find again a symmetry if $\gamma=0$, as for
  Corollaries~\ref{cor:cvDN2cla} and \ref{cor:cvND2cla}, and with
  $\theta=1$, ND$_3$ diverges like DN$_2$ when $\gamma=0$. In fact, in
  this case, the convergence factor of ND$_3$~\eqref{eq:rhoND_3gam0}
  is very similar to the convergence factor of
  DN$_2$~\eqref{eq:rhoDN_2}. Due to this divergence, we cannot
provide a general estimate of the convergence factor. We can
  however still study the convergence behavior for extreme
eigenvalues. Inserting $d_i=0$ into~\eqref{eq:rhoND3}, we find also
\eqref{eq:DN2d0}, and thus for small eigenvalues ND$_3$ behaves
  like DN$_2$, like we observed for ND$_2$ and DN$_3$ earlier. When
the eigenvalue goes to infinity, we also obtain
 \[\lim_{\d_i\rightarrow\infty}\rho_{\text{ND}_3} = |1-2\theta|.\]
%$\lim_{\d_i\rightarrow\infty}\rho_{\text{ND}_3} = |1-2\theta|$.
Hence all the four algorithms DN$_2$, ND$_2$, DN$_3$ and ND$_3$
  have the same limit for large eigenvalues. By
  equioscillation, we then obtain the same relaxation parameter
  as~\eqref{eq:thetaDN2opt}. This leads to a similar result as
  Theorem~\ref{thm:thetaDN2opt}.

\begin{theorem}\label{thm:thetaND3opt}
  If we assume that the eigenvalues of $A$ are anywhere in the
  interval $[0,\infty)$, then the optimal relaxation parameter
    $\theta^{\star}_{\text{ND}_3}$ for the algorithm
    ND$_3$~\eqref{eq:ND3}-\eqref{eq:ND3tran} with $\gamma=0$ is
    identical to $\theta^{\star}_{\text{DN}_2}$.
\end{theorem}

\begin{proof}
  In the case $\gamma=0$, the convergence factors~\eqref{eq:rhoDN_2}
  and~\eqref{eq:rhoND_3gam0} are the same when exchanging $a_i$
  and $b_i$, and thus the proof follows as for
  Theorem~\ref{thm:thetaDN2opt}.
\end{proof}

\section{Numerical experiments}\label{sec:4}

We illustrate now our six new time domain decomposition algorithms with
  numerical experiments. We divide the time domain
  $\Omega=(0,1)$ into two non-overlapping subdomains with
  interface $\alpha$, and fix the regularization parameter $\nu=0.1$.
We will investigate the performance by plotting the convergence
  factor as function of the eigenvalues $d\in[10^{-2},10^{2}]$.

\subsection{Convergence factor with \texorpdfstring{$\theta=1$}{th1} for a symmetric decomposition}

 We show in Figure~\ref{fig:cv6_th1_sym} the convergence factors
  for all six algorithms for a symmetric decomposition,
  $\alpha=\frac12$, with $\theta=1$, on the left
\begin{figure}
  \centering
  \includegraphics[scale=0.15]{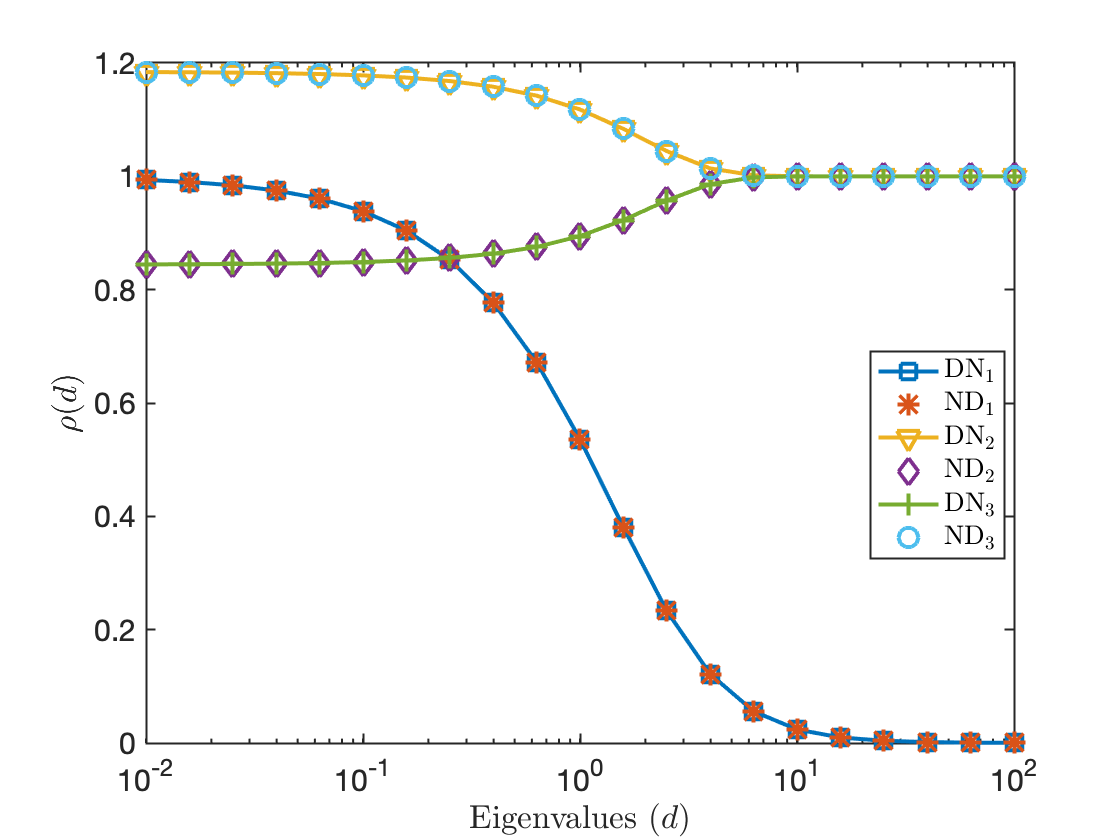}
  \includegraphics[scale=0.15]{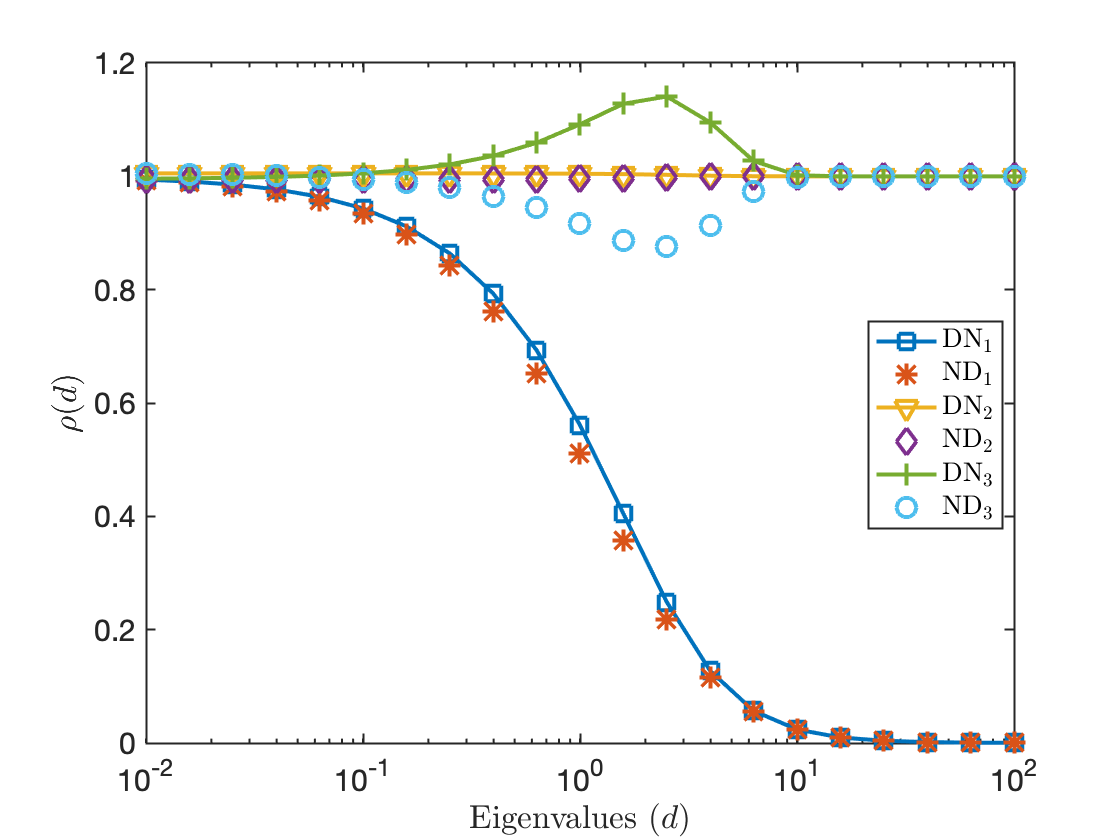}
  \caption{Convergence factor with $\theta=1$ for a symmetric
    decomposition of the six new algorithms as function of the eigenvalues
    $d\in[10^{-2},10^{2}]$. Left: $\gamma=0$. Right: $\gamma=10$.}
  \label{fig:cv6_th1_sym}
\end{figure}
without final target state (i.e., $\gamma=0$), and on the right with
a final target state for $\gamma=10$.  Without final target state,
  the convergence factor of DN$_1$ and ND$_1$ coincide, as one can see
  also by substituting $\gamma=0$ and $a_i=b_i$
into~\eqref{eq:rhoDN_1} and~\eqref{eq:rhoND_1}. The same also
  holds for the pairs DN$_2$ and ND$_3$, and DN$_3$ and ND$_2$. We
also see the symmetry between DN$_2$ and ND$_2$, as well as DN$_3$ and
ND$_3$. This changes when a final target state with $\gamma=10$
is present: while the convergence behavior remains similar for
  DN$_1$ and ND$_1$, the symmetry between DN$_2$ and
  ND$_2$\footnote{This is a bit hard to see on the right in
    Figure~\ref{fig:cv6_th1_sym}, but zooming in confirms that the
    convergence factor of DN$_2$ is above 1, and below 1 for ND$_2$.}
 and DN$_3$ and ND$_3$ remains. Furthermore, DN$_3$ converges with
no final target but diverges with $\gamma=10$, and vice versa for
ND$_3$. In terms of the convergence speed, DN$_1$ and ND$_1$ are
  much better than the other four algorithms for high frequencies in
  both cases, and ND$_1$ is slightly better overall than DN$_1$ when
  $\gamma=10$. The good high frequency behavior follows from our
  analysis: it depends for all 6 algorithms only on $\theta$. In the
  case $\theta=1$ here, the limit is $|1-\theta|=0$ for DN$_1$ and
  ND$_1$, and $|1-2\theta|=1$ for DN$_2$, DN$_3$, ND$_2$ and
  ND$_3$. For the zero frequency, $d=0$, the convergence factor for
  DN$_1$ and ND$_1$ equals 1 for all $\gamma$, but for DN$_2$, DN$_3$,
  ND$_2$ and ND$_3$ this depends on $\gamma$. Inserting $\theta=1$
  into~\eqref{eq:DN2d0} and~\eqref{eq:ND2d0}, we obtain for DN$_2$ and
  ND$_3$ the convergence factor
  $\coth(\sqrt{\nu^{-1}}\alpha)\frac{\sqrt{\nu^{-1}}\coth(\sqrt{\nu^{-1}}\alpha)
    + \nu^{-1}\gamma}{\sqrt{\nu^{-1}} +
    \nu^{-1}\gamma\coth(\sqrt{\nu^{-1}}\alpha)}$, and for ND$_2$ and
  DN$_3$
  $\tanh(\sqrt{\nu^{-1}}\alpha)\frac{\sqrt{\nu^{-1}}\tanh(\sqrt{\nu^{-1}}\alpha)
    + \nu^{-1}\gamma}{\sqrt{\nu^{-1}} +
    \nu^{-1}\gamma\tanh(\sqrt{\nu^{-1}}\alpha)}$. For $\gamma=0$, the
  two convergence factors are approximately 1.185 for DN$_2$ and
  ND$_3$, 0.844 for ND$_2$ and DN$_3$, and for $\gamma=10$, we get
  1.005 for DN$_2$ and ND$_3$, and 0.995 ND$_2$ and DN$_3$.

\subsection{Convergence factor with \texorpdfstring{$\theta=1$}{th1} for an asymmetric decomposition}

For $\theta=1$, we show on the left in
Figure~\ref{fig:cv6_th1_asym} the convergence factors with interface
at $\alpha=0.3$ and no final target state (i.e., $\gamma=0$), and on
the right $\alpha=0.7$ with a final target state $\gamma=10$.
\begin{figure}
  \centering
  \includegraphics[scale=0.15]{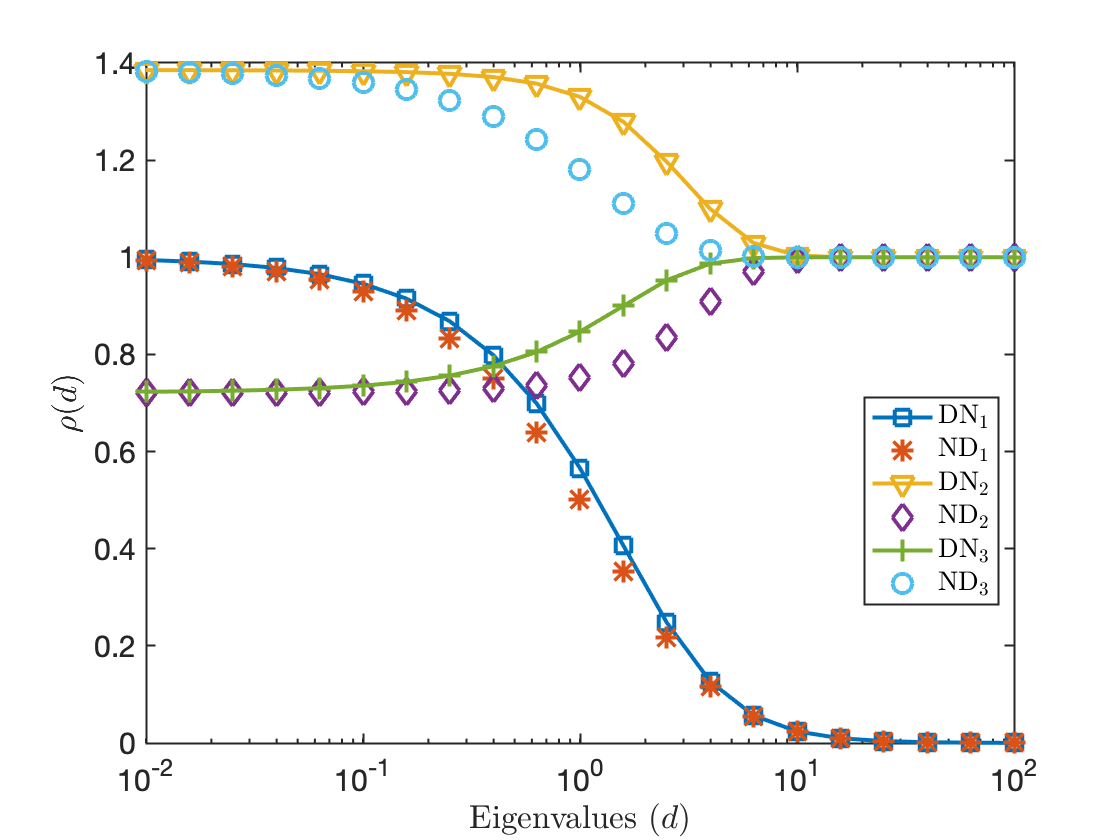}
  \includegraphics[scale=0.15]{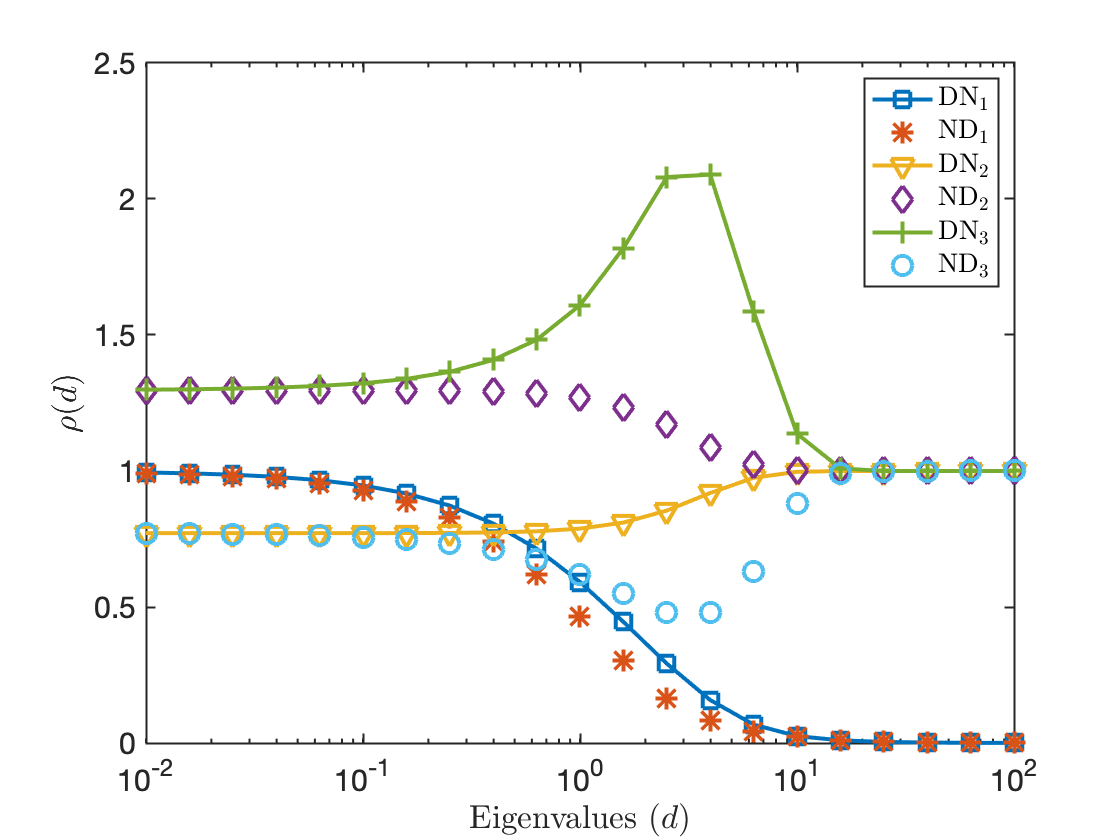}
  \caption{Convergence factor with $\theta=1$ for an asymmetric
    decomposition of all six new algorithms as function of the eigenvalues
    $d\in[10^{-2},10^{2}]$. Left: $\gamma=0$ and $\alpha=0.3$. Right:
    $\gamma=10$ and $\alpha=0.7$.}
  \label{fig:cv6_th1_asym}
\end{figure}
For DN$_1$ and ND$_1$, the convergence factor is similar in both
cases, ND$_1$ being slightly better, and convergence is also similar
to the symmetric case. This is because the convergence factor of the
two algorithms for small and large eigenvalues is independent of the
values of $\alpha,\nu$ and $\gamma$. Their high frequency behavior
  is also much better compared to the other four algorithms in the two
  cases.  For the other four algorithms, we see again the symmetry
  between DN$_2$ and ND$_2$, and DN$_3$ and ND$_3$. In general, DN$_2$
  and ND$_3$ behave similarly, and also ND$_2$ and DN$_3$, but the
  influence of $\gamma$ is more significant for DN$_3$ and ND$_3$ than
  DN$_2$ and ND$_2$. However their convergence factors all go to 1 for
  large eigenvalues, as for the symmetric decomposition.
  For the zero frequency, using the expressions~\eqref{eq:DN2d0}
  and~\eqref{eq:ND2d0} with $\theta=1$, we obtain approximately 1.386 
  for DN$_2$ and ND$_3$, and 0.722 for ND$_2$ and DN$_3$ in the 
  case $\gamma=0$, $\alpha=0.3$. For  $\gamma=10$, $\alpha=0.7$, 
  we get 0.771 for DN$_2$ and ND$_3$, and 1.296 for ND$_2$ and DN$_3$.

\subsection{Convergence factor for Category I with different \texorpdfstring{$\theta$}{th}}

Since DN$_1$ and ND$_1$ performed quite similarly, and much better
than the others, we now investigate the dependence of DN$_1$ on
$\theta$ in more detail. On the left in Figure~\ref{fig:cv1_5th}
\begin{figure}
  \centering
  \includegraphics[scale=0.15]{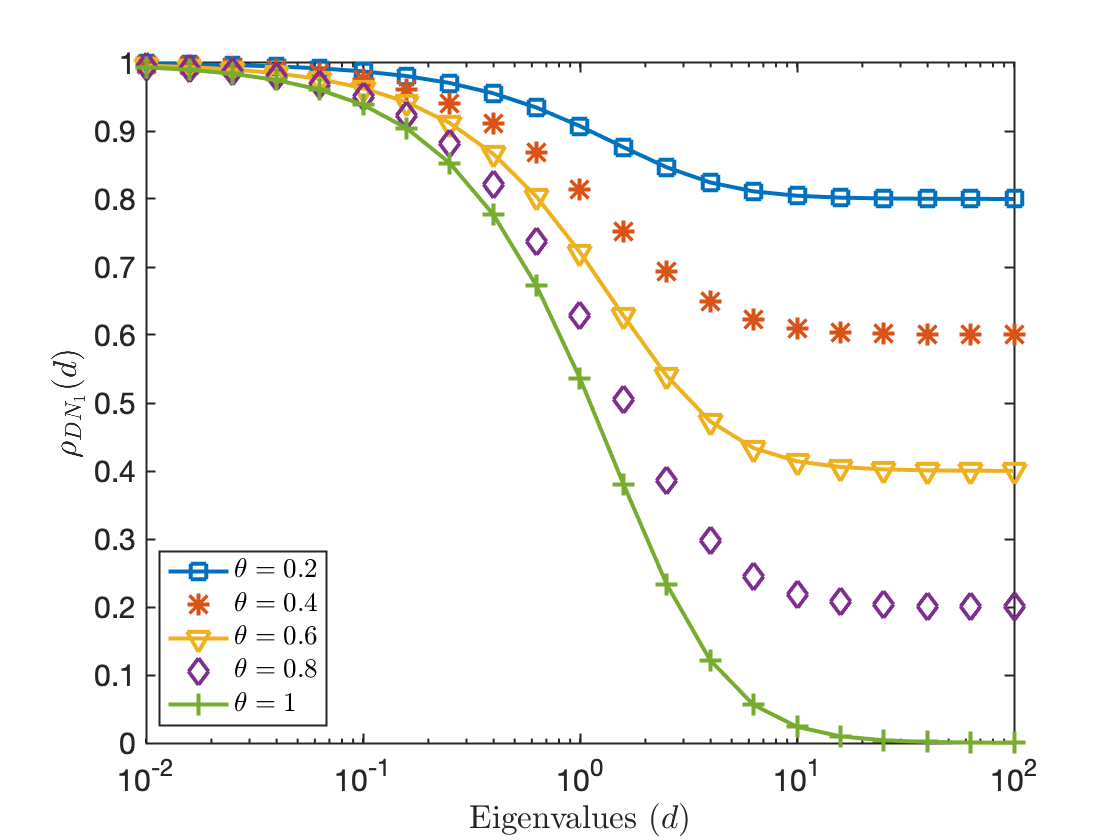}
  \includegraphics[scale=0.15]{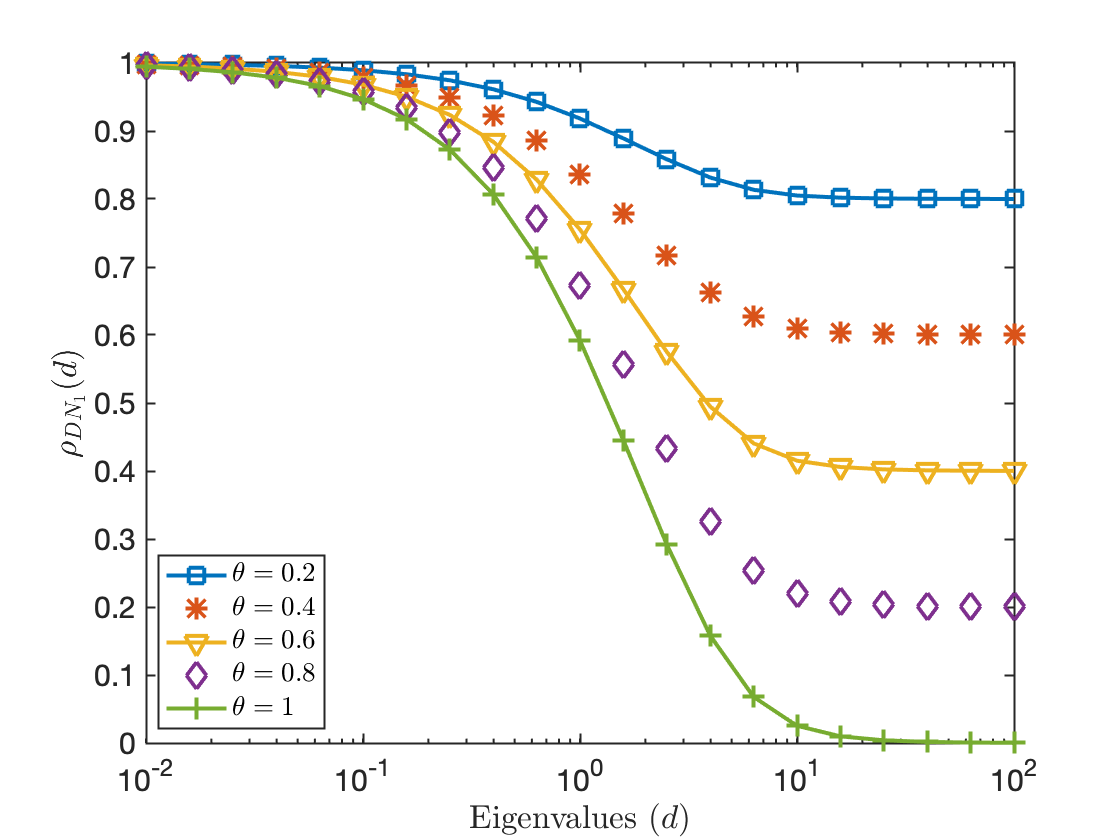}
  \caption{Convergence factor with different relaxation parameters of
    DN$_1$ as function of the eigenvalues
    $d\in[10^{-2},10^{2}]$. Left: $\gamma=0$ and $\alpha=0.5$. Right:
    $\gamma=10$ and $\alpha=0.7$.}
  \label{fig:cv1_5th}
\end{figure}
we show the convergence factor of DN$_1$ without final target
  state and a symmetric decomposition, and on the right with a final
  target state $\gamma=10$ and an asymmetric decomposition. The
  convergence is very similar for these two settings, DN$_1$ is
  robust, and $\theta=1$ gives the best performance.

\subsection{Convergence factor with optimal \texorpdfstring{$\theta$}{thopt} for a symmetric decomposition}

Since the algorithms in Categories II and III are strongly related, we
  compare them now in Figure~\ref{fig:cv4_thopt_sym} for a symmetric
  decomposition using their optimal relaxation parameter
  $\theta^{\star}$, obtained numerically.
\begin{figure}
  \centering
  \includegraphics[scale=0.15]{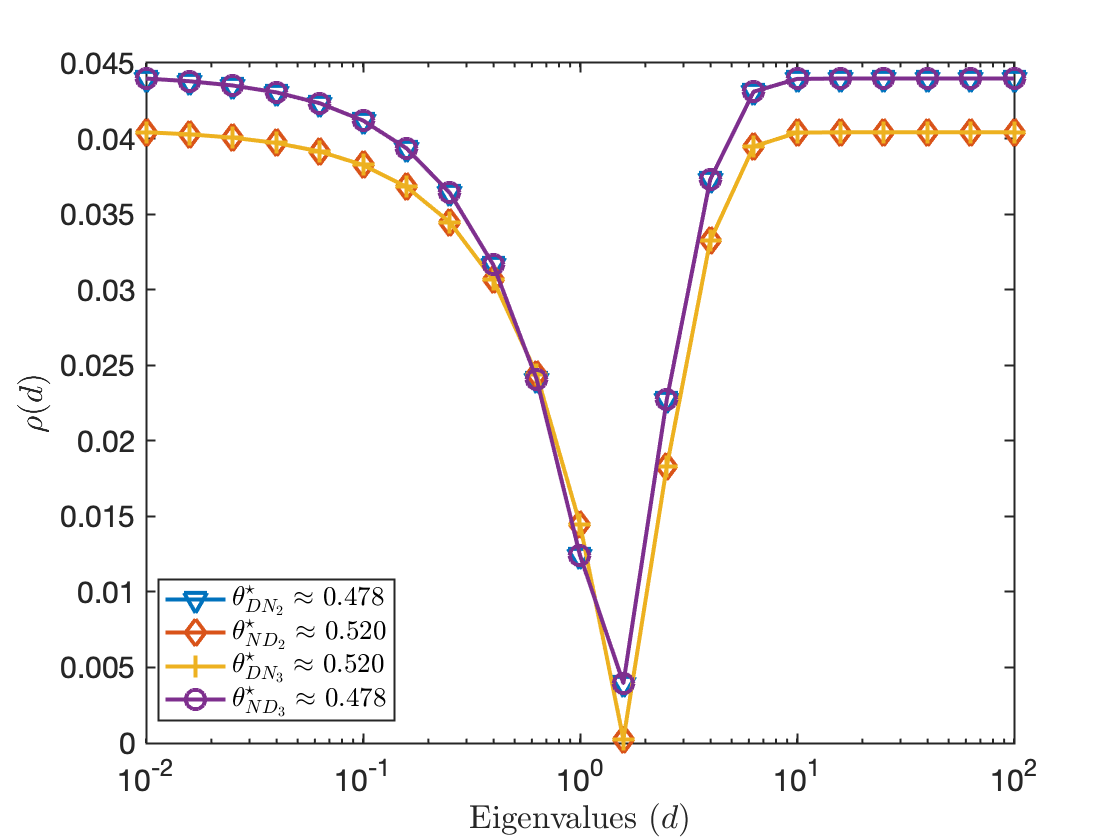}
  \includegraphics[scale=0.15]{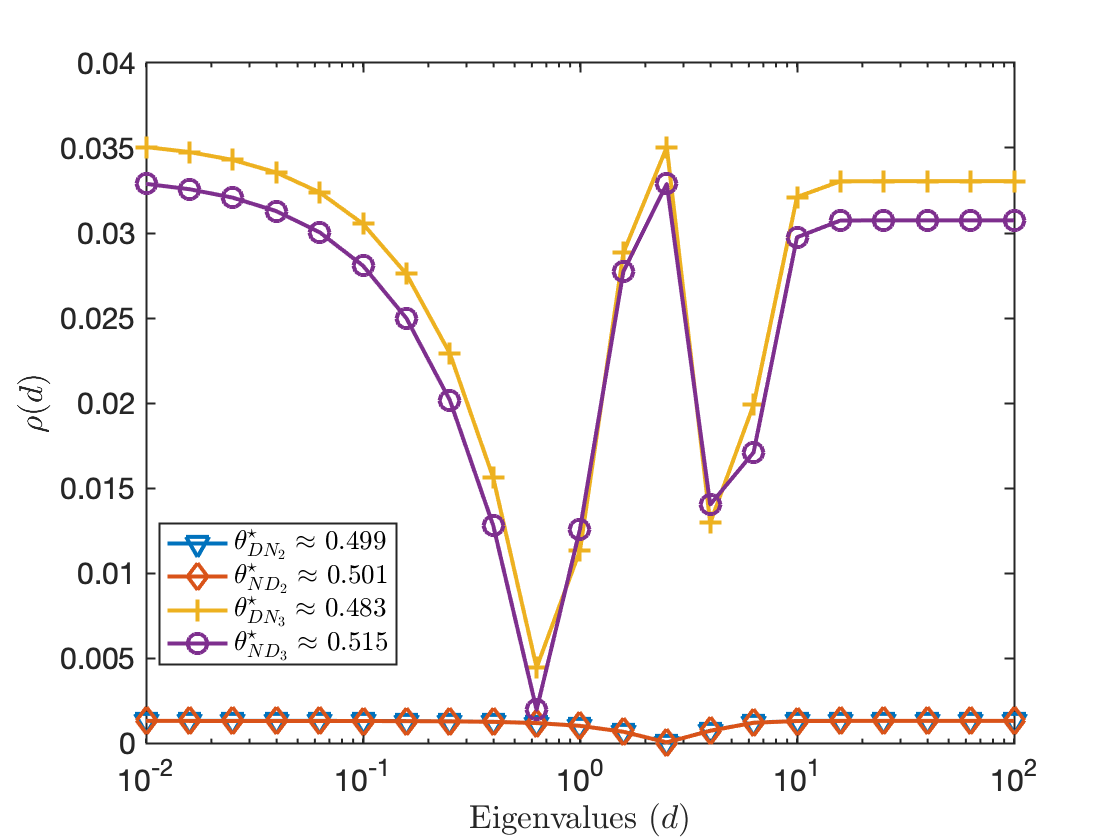}
  \caption{Convergence factor with $\theta^{\star}$ for a symmetric
    decomposition as function of the eigenvalues
    $d\in[10^{-2},10^{2}]$. Left: $\gamma=0$. Right: $\gamma=10$.}
  \label{fig:cv4_thopt_sym}
\end{figure}
On the left without final state, DN$_2$ and ND$_3$, and also
  ND$_2$ and DN$_3$, have the same convergence factor, and the
optimal relaxation parameter satisfies
$\theta^{\star}_{\text{DN}_2}=\theta^{\star}_{\text{ND}_3}$ and
$\theta^{\star}_{\text{ND}_2}=\theta^{\star}_{\text{DN}_3}$ as proved
in Theorem~\ref{thm:thetaDN3opt} and
Theorem~\ref{thm:thetaND3opt}. These correspond to the value found
using~\eqref{eq:thetaDN2opt} and~\eqref{eq:thetaND2opt}. In terms of
the convergence speed, ND$_2$ and DN$_3$ are slightly better than
DN$_2$ and ND$_3$. However, these similarities disappear when we add a
final target state $\gamma=10$. On the right in
  Figure~\ref{fig:cv4_thopt_sym}, we see that now the convergence
  behavior of DN$_2$ and ND$_2$ is similar, and also DN$_3$ and
  ND$_3$ are rather similar, and DN$_2$ and ND$_2$ converge much
  faster compared to the others. We also see equioscillation between
small and large eigenvalues. The theoretical results
in~\eqref{eq:thetaDN2opt} as well as in~\eqref{eq:thetaND2opt} still
determine the optimal relaxation parameter
$\theta^{\star}_{\text{DN}_2}$ and $\theta^{\star}_{\text{ND}_2}$
for DN$_2$ and ND$_2$, but not for DN$_3$ and ND$_3$, where we observe
an equioscillation between small eigenvalues with some eigenvalues in
the interval $[1,10]$. Also ND$_3$ is
slightly better than DN$_3$.

\subsection{Convergence factor with optimal \texorpdfstring{$\theta$}{thopt} for an asymmetric decomposition}

We show in Figure~\ref{fig:cv4_thopt_asym} the convergence factor
with the optimal relaxation parameter $\theta^{\star}$ for the four
algorithms in Categories II and III for an asymmetric decomposition.
\begin{figure}
  \centering
  \includegraphics[scale=0.15]{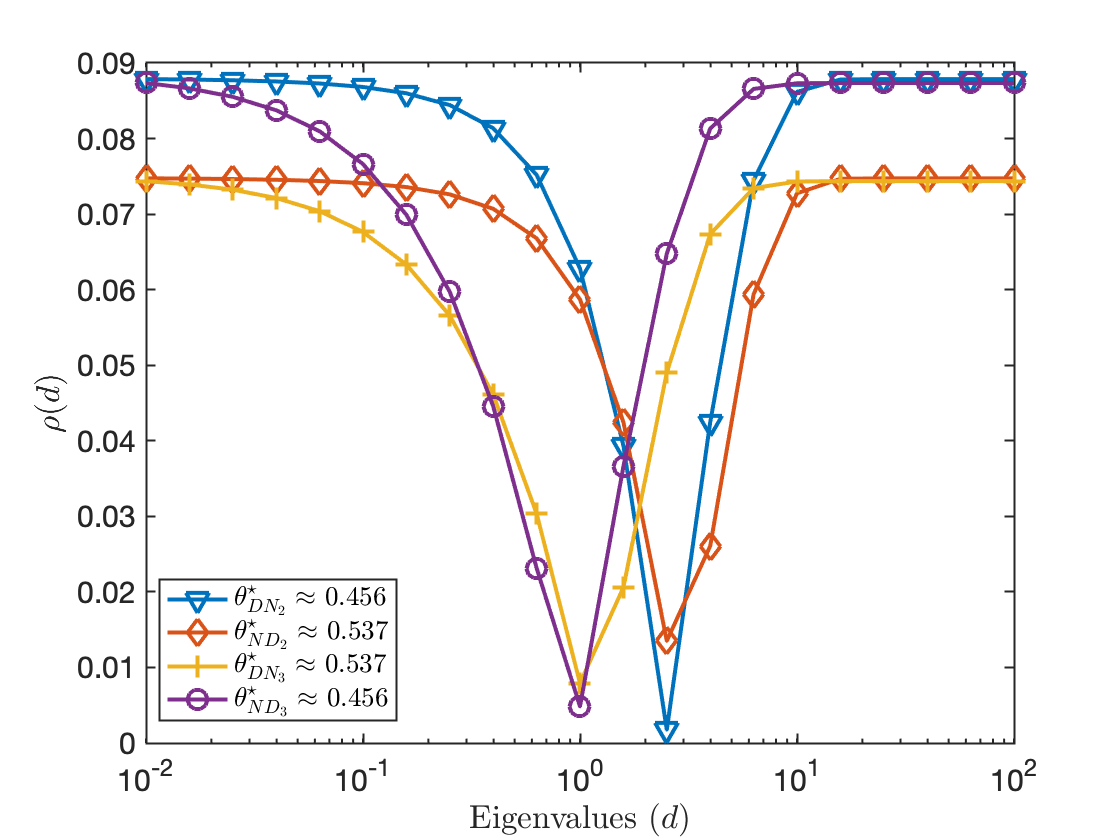}
  \includegraphics[scale=0.15]{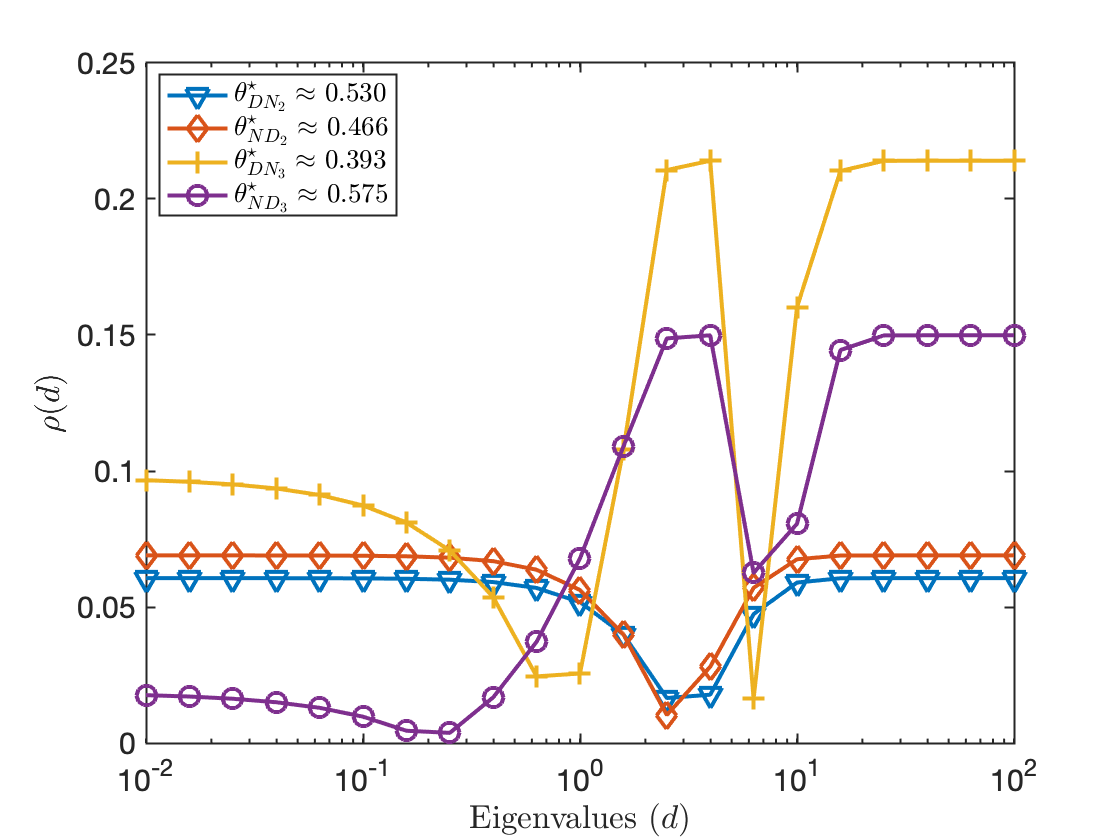}
  \caption{Convergence factor with $\theta^{\star}$ for an asymmetric
    decomposition as function of the eigenvalues
    $d\in[10^{-2},10^{2}]$. Left: $\gamma=0$ and $\alpha=0.3$. Right:
    $\gamma=10$ and $\alpha=0.7$.}
  \label{fig:cv4_thopt_asym}
\end{figure}
  On the left with $\alpha=0.3$ and no target state $\gamma=0$ the
  convergence factors of the four algorithms are similar. This is
  consistent with the monotonicity we proved without final
state. The optimal relaxation parameters satisfy
$\theta^{\star}_{\text{DN}_2}=\theta^{\star}_{\text{ND}_3}$ and
$\theta^{\star}_{\text{ND}_2}=\theta^{\star}_{\text{DN}_3}$, and we
can use~\eqref{eq:thetaDN2opt} and~\eqref{eq:thetaND2opt} to determine
their values. Similar to the symmetric decomposition, ND$_2$ and
DN$_3$ are slightly better than the others. However, these properties
disappear again on the right in Figure~\ref{fig:cv4_thopt_asym}
when there is a final state $\gamma=10$.  While DN$_2$ and
  ND$_2$ still equioscillate between the small and large eigenvalues,
  and the optimal relaxation parameter can be determined
using~\eqref{eq:thetaDN2opt} and~\eqref{eq:thetaND2opt}, for
  DN$_3$ and ND$_3$ the equioscillation is between large eigenvalues
  and some eigenvalues in the interval $[1,10]$. Hence, the optimal
relaxation parameters for the algorithms DN$_3$ and ND$_3$ are
different from DN$_2$ and ND$_2$. Also DN$_2$ and ND$_2$ converge much
faster than the other two, and DN$_2$ is slightly faster than ND$_2$.

\section{Conclusion}\label{sec:5}

  We introduced and analyzed six new time domain decomposition
  methods based on Dirichlet-Neumann and Neumann-Dirichlet techniques
  for parabolic optimal control problems. Our analysis shows that
  while at first sight it might be natural to preserve the
  forward-backward structure in the time subdomains as well, there are
  better choices that lead to substantially faster algorithms. We
  find that the algorithms in Categories II and III with
  optimized relaxation parameter are much faster than the algorithms
  in Category I, and they can still be identified to be of
  forward-backward structure using changes of variables. We also found
  many interesting mathematical connections between these
  algorithms. Algorithms in Category I are natural smoothers, while
  algorithms in Categories II and III with optimized relaxation
  parameter are highly efficient solvers.

  Our study was restricted to the two subdomain case, but the
  algorithms can all naturally be written for many subdomains, and
  then one can also run them in parallel. They can also be used for
  more general parabolic constraints than the heat equation. Extensive
  numerical results will appear elsewhere.

\bibliographystyle{siamplain}
\bibliography{references}

\appendix

\section{Convergence analysis using \texorpdfstring{$\mu_{(i)}$}{ui}}\label{sec:app1}

We can also use formulation~\eqref{eq:mu} to analyze the
convergence behavior of the
algorithm DN$_1$~\eqref{eq:DN1}-\eqref{eq:DN1tran}, we then need to study
\begin{equation}\label{eq:errDN1bis}
  \left\{
    \begin{aligned}
      \ddot \mu_{1,(i)}^k - \sigma_i^2 \mu_{1,(i)}^k &= 0 \text{ in } \Omega_1,\\
      \dot \mu_{(i)}(0)-d_i\mu_{(i)}(0) &= 0, \\
      \mu_{1,(i)}^k(\alpha) &= f_{\alpha,(i)}^{k-1},
    \end{aligned}
  \right.
  \quad
  \left\{
    \begin{aligned}
      \ddot \mu_{2,(i)}^k - \sigma_i^2 \mu_{2,(i)}^k &= 0 \text{ in } \Omega_2,\\
      \ddot \mu_{2,(i)}^k(\alpha) - d_i\dot\mu_{2,(i)}^k(\alpha) &= \ddot \mu_{1,(i)}^k(\alpha) - d_i\dot\mu_{1,(i)}^k(\alpha),\\
      \gamma \dot \mu_{(i)}(T)+\beta_i\mu_{(i)}(T)&=0,
    \end{aligned}
  \right.
\end{equation}
with the update of the transmission condition
\begin{equation}\label{eq:errDN1tranbis}
  f_{\alpha,(i)}^{k}= (1-\theta)f_{\alpha,(i)}^{k-1} + \theta\mu_{2,(i)}^k(\alpha)\quad \theta\in(0,1). 
\end{equation}
This is a DR type algorithm applied to
solve~\eqref{eq:mu}. Using~\eqref{eq:musol}, we determine the two
coefficients $A_i^k$ and $B_i^k$ from the transmission condition
from~\eqref{eq:errDN1bis}.
% , we obtain
% \[\begin{aligned}
%   A_i^k&=\frac{f_{\alpha,(i)}^{k-1}}{\sigma_i\cosh(a_i)+d_i\sinh(a_i)},\\
%   B_i^k&=\frac{\nu^{-1}f_{\alpha,(i)}^{k-1}\cosh(a_i)}{\big(\sigma_i\cosh(a_i)+d_i\sinh(a_i)\big)\big(\omega_i\cosh(b_i)+\sigma_i\sinh(b_i)\big)}.
% \end{aligned}\]
Using then relation~\eqref{eq:errDN1tranbis}, we find
\[f_{\alpha,(i)}^{k}= (1-\theta)f_{\alpha,(i)}^{k-1} + \theta\nu^{-1}f_{\alpha,(i)}^{k-1}\frac{\gamma\sigma_i+\beta_i\tanh(b_i)}{\big(\sigma_i+d_i\tanh(a_i)\big)\big(\omega_i+\sigma_i\tanh(b_i)\big)},\]
which is exactly the same convergence factor as~\eqref{eq:rhoDN1}.
\end{document}